\documentclass{article}

\usepackage{microtype}
\usepackage{graphicx}
\usepackage{subfigure}
\usepackage{booktabs} 

\usepackage{mathtools}

\usepackage{hyperref}

\usepackage[accepted]{icml2023}

\usepackage[utf8]{inputenc} 
\usepackage[T1]{fontenc}    
\usepackage{hyperref}       
\usepackage{booktabs}       
\usepackage{amsfonts}       
\usepackage{nicefrac}       
\usepackage{microtype}      
\usepackage{xcolor}         
\usepackage[title]{appendix}
\usepackage{todonotes}

\usepackage{graphicx,verbatim,lineno,titletoc}
\usepackage{amssymb}
\usepackage{amsmath, amsthm}
\usepackage[all]{xy}
\usepackage{enumerate}
\usepackage[english]{babel}
\usepackage{multirow}
\usepackage{float}
\usepackage{enumitem}
\usepackage{accents}

\usepackage[capitalize,noabbrev]{cleveref}

\theoremstyle{plain}
\newtheorem{theorem}{Theorem}[section]
\newtheorem{proposition}[theorem]{Proposition}
\newtheorem{lemma}[theorem]{Lemma}
\newtheorem{corollary}[theorem]{Corollary}
\theoremstyle{definition}

\newtheorem{assumption}[theorem]{Assumption}
\theoremstyle{remark}
\newtheorem{remark}[theorem]{Remark}

\newcommand{\commHL}[1]{{\textcolor{blue}{#1}}}

\usepackage{pifont}
\newcommand{\xmark}{\ding{55}}%

\def\R{\mathbb{R}}

\def\E{\mathbb{E}}
\def\P{\mathbb{P}}

\def\H{\mathbf{H}}

\def\eps{\varepsilon}

\def\X{\mathbf{X}}
\def\x{\mathbf{x}}

\def\param{\boldsymbol{\theta}}
\def\Param{\boldsymbol{\Theta}}

\DeclareMathOperator{\prox}{\textup{prox}}

\DeclareMathOperator{\proj}{\textup{proj}}

\DeclareMathOperator{\dist}{dist}

\DeclareMathOperator*{\argmin}{arg\,min}

\newcommand{\tr}{\textup{tr}}

\def\R{\mathbb{R}}
\def\E{\mathbb{E}}

\def\P{\mathbb{P}}

\def\H{\mathbf{H}}

\def\eps{\varepsilon}

\def\X{\mathbf{X}}
\def\x{\mathbf{x}}

\def\H{\mathbf{H}}

\icmltitlerunning{Convergence of First-Order Methods with Dependent Data}

\begin{document}
	
	\twocolumn[
	
	\icmltitle{Convergence of First-Order Methods for Constrained Nonconvex Optimization \\ with Dependent Data}

	
	
	\icmlsetsymbol{equal}{*}
	
	\begin{icmlauthorlist}
		\icmlauthor{Ahmet Alacaoglu}{equal,yyy}
		\icmlauthor{Hanbaek Lyu}{equal,xxx}
	\end{icmlauthorlist}

	\icmlaffiliation{xxx}{Department of Mathematics, University of Wisconsin--Madison, WI, USA}
	
	\icmlaffiliation{yyy}{Wisconsin Institute for Discovery, University of Wisconsin--Madison, WI, USA}

	
	\icmlcorrespondingauthor{Hanbaek Lyu}{hlyu@math.wisc.edu}

	\icmlkeywords{First-order methods, nonconvex optimization, convergence, complexity, dependent data, Markov chains}
	
	\vskip 0.3in
	]
	
	
	
	\printAffiliationsAndNotice{\icmlEqualContribution} 
	
	\begin{abstract}
We focus on analyzing the classical stochastic projected gradient methods under a general dependent data sampling scheme for constrained smooth nonconvex optimization. We show the worst-case rate of convergence $\tilde{O}(t^{-1/4})$ and complexity $\tilde{O}(\varepsilon^{-4})$ for achieving an $\varepsilon$-near stationary point in terms of the norm of the gradient of Moreau envelope and gradient mapping. While classical convergence guarantee requires i.i.d. data sampling from the target distribution, we only require a mild mixing condition of the conditional distribution, which holds for a wide class of Markov chain sampling algorithms. This improves the existing complexity for the constrained smooth nonconvex optimization with dependent data from $\tilde{O}(\varepsilon^{-8})$ to $\tilde{O}(\varepsilon^{-4})$ with a significantly simpler analysis. We illustrate the generality of our approach by deriving convergence results with dependent data for stochastic proximal gradient methods, adaptive stochastic gradient algorithm AdaGrad and stochastic gradient algorithm with heavy ball momentum. As an application, we obtain first online nonnegative matrix factorization algorithms for dependent data based on stochastic projected gradient methods with adaptive step sizes and optimal rate of convergence.
	\end{abstract}

	\section{Introduction}
	\label{Introduction}
 
	Consider the minimization of a function $f:\R^{p}\rightarrow \mathbb{R}$ given as an expectation:
	\begin{equation}\label{eq:expected_loss_minimization}
		\param^{*} \in \argmin_{\param\in \Param} \left\{f(\param):= \E_{\x\sim \pi}\left[ \ell(\param, \x)  \right]\right\},
	\end{equation} 
	where $\pi$ is a distribution on a sample space $\Omega\subseteq \R^{d}$ with a density function; $\ell\colon\Omega\times \Param\rightarrow \mathbb{R}$ a per-sample loss function and $\Param\subseteq \R^{p}$ a closed convex set with an efficiently computable projection
	\begin{equation}
		\proj(\param) = \argmin_{\param'\in\Param} \frac{1}{2} \lVert \param-\param' \rVert^{2}.
	\end{equation}
	We assume that $f$ is a smooth and possibly nonconvex function.
	Constrained nonconvex optimization with \emph{dependent data} arise in many situations such as decentralized constrained optimization over networked systems, where the i.i.d. sampling requires significantly more communication than the dependent sampling~\citep{johansson2007simple, johansson2010randomized, duchi2012ergodic}.
	Other applications are policy evaluation in reinforcement learning where the Markovian data is naturally present since the underlying model is a Markov Decision Process~\citep{bhandari2018finite}, and  online nonnegative matrix factorization and network denoising~\citep{lyu2020online}.
	
	\subsection{Related Work and Summary of Contributions}
	It is well-known that obtaining optimal complexity with \emph{single-sample} projected stochastic gradient descent (SGD) for constrained nonconvex problems is significantly more challenging than unconstrained nonconvex optimization~\citep{ghadimi2016mini,davis2019stochastic, alacaoglu2020convergence}.
	This challenge has been recently overcome by~\citep{davis2019stochastic} within the framework of weakly convex optimization, which resulted in optimal complexity results for projected/proximal SGD (PSGD).
	Later, this result is extended for algorithms such as SGD with heavy ball momentum~\citep{mai2020convergence} or adaptive algorithms such as AMSGrad and AdaGrad~\citep{alacaoglu2020convergence}.
	These guarantees require i.i.d. sampling from the underlying distribution $\pi$.

	Optimization with non-i.i.d. data is studied in the convex and nonconvex cases with gradient/mirror descent in~\citep{sun2018markov,duchi2012ergodic,nagaraj2020least} and block coordinate descent in~\citep{sun2020markov}. SGD is also recently considered in~\citep{wang2021sample} for convex problems. 
Another important work in this direction is~\citep{karimi2019non} that focused on unconstrained nonconvex case with a different assumption on the dependent data compared to previous works and relaxed the assumptions on the variance.
	For constrained and nonconvex problems, the work~\citep{lyu2020online} showed asymptotic guarantees of stochastic majorization-minimization (SMM)-type algorithms to stationary points of the expected loss function.

	More recently, \citep{lyu2022convergence} studied generalized SMM-type algorithms with dependent data for constrained problems and showed the complexity $\tilde{O}(\varepsilon^{-8})$ with standard assumptions (that we will clarify in the sequel) and $\tilde{O}(\varepsilon^{-4})$ when all the iterates of the algorithm lie in the interior of the constraint set, for obtaining $\varepsilon$-stationarity.
	We also remark that \citep{lyu2022convergence} showed that for the `empirical loss functions' (recursive average of sample losses), SMM-type algorithms need only $\tilde{O}(\eps^{-4})$ iterations for making the stationarity gap under $\eps$. Our present work does not consider empirical loss functions but focus on expected loss functions. See \citep{lyu2022convergence} for more details. 
	
	Since the complexity $\tilde O(\varepsilon^{-8})$ is suboptimal for nonconvex expected loss minimization, the motivation of our work is to understand if this complexity is improvable or if it is inherent when we handle dependent data and constraints jointly.
	Our results conclude that the complexity is indeed improvable and show the near-optimal complexity $\tilde{O}(\varepsilon^{-4})$ for constrained nonconvex problems with dependent data.
	Unlike our result, previous work~\citep{lyu2022convergence} needed an additional assumption that the iterates lie in the interior of the constraint (which is difficult to satisfy in general for constrained problems) for the optimal complexity $\tilde{O}(\varepsilon^{-4})$.
	Moreover, to our knowledge, no convergence rate of \emph{projected} SGD is known in the \emph{constrained} nonconvex case with non-i.i.d. sampling.
	We also show the first rates for AdaGrad~\citep{Duchi:EECS-2010-24} and SGD with heavy ball momentum~\citep{mai2020convergence} for this setting. See Table \ref{table: t1} for a summary of the discussion above. 
	
	After the completion of our manuscript, we became aware of the recent concurrent work~\citep{dorfman2022adapting} that analyzed AdaGrad with multi level Monte Carlo gradient estimation for dependent data.
	This work focused on the \emph{unconstrained} nonconvex setting whereas our main focus is the more general class of \emph{constrained} nonconvex problems.
	Hence we believe the two results complement each other.

    {We also note that slightly stronger versions of Assumption~\ref{assumption:A2} are required even for unconstrained nonconvex optimization with dependent data, see~\cite{sun2018markov,dorfman2022adapting}. It is well-known that this assumption is difficult to satisfy in the unconstrained setting, but it is more realistic 
    with the presence of constraints. Because of this reason, our results incorporating the constraints and projections in the algorithm provides a more realistic problem setup. 
    While our results would recover those in~\cite{sun2018markov} when specialized to the unconstrained case,     due to Assumption~\ref{assumption:A2}, this unconstrained setting would be less realistic as argued above. Because of this, and for other motivating applications, the main focus of this paper is obtaining optimal complexity results for constrained nonconvex problems.}
	
	\subsection{Contribution}
	
	We consider convergence of stochastic first-order methods, including proximal and projected stochastic gradient descent (SGD), projected SGD with momentum, and stochastic adaptive gradient descent (AdaGrad-norm). These are all classical nonconvex optimization algorithms that have been used extensively in various optimization and machine learning tasks. Our main focus is to establish optimal convergence rate for such stochastic first-order methods under very general data sampling scheme, including functions of Markov chains, state-dependent Markov chains, and more general stochastic processes with fast enough mixing of multi-step conditional distribution. 
	
	To summarize our results, consider the following simple first-order method:
	\begin{description}
		\item[Step 1.] Sample $\x_{t+1}$ from a distribution conditional on $\x_{1},\dots,\x_{t}$; ($\triangleright$ possibly non-i.i.d. samples) 
		\item[Step 2.] Compute a stochastic gradient $G(\param_{t}, \x_{t+1})$ (see Assumption \ref{assumption:A2} for Def.) and  $\param_{t+1} \leftarrow \textup{proj}_{\Param} \left(  \param_{t} - \alpha_{t} G(\param_{t}, \x_{t+1}) \right)$, where the step size $\alpha_{t}$ is chosen so that either (1) non-summable and square-summable; or (2) according to AdaGrad-norm: $\alpha_{t}^{-2}=\alpha_{t-1}^{-2}+\lVert G(\param_{t}, \x_{t+1}) \rVert^{2}\alpha^{-2}$ for $\alpha>0$.
	\end{description}
	
	An important point here is that \textit{we do not require the new training point $\x_{t+1}$ to be distributed according to the stationary distribution $\pi$, nor to be independent on all the previous samples $\x_{1},\dots,\x_{t}$.} For instance, we allow one to sample $\x_{t+1}$ according to an underlying Markov chain, so that each step of sampling is computationally very efficient but the distribution $\x_{t+1}$ conditional on $\x_{t}$ could be far from $\pi$. This may induce bias in estimating the stochastic gradient $G(\param_{t}, \x_{t+1})$. 
	
	Suppose $f$ is $\rho$-smooth; $\Param\subseteq \R^{p}$ is convex, closed; and the training samples $\x_{t}$ are a function of some underlying Markov chain mixing sufficiently fast (see Section~\ref{sec: grad_estim}). Under some mild assumptions used in the literature
    \citep{sun2018markov,lyu2022convergence,bhandari2018finite}, we establish the following convergence results for a wide range of stochastic first-order methods under non-i.i.d. data setting:
	
	\begin{itemize}
	    \item We show that any convergent subsequence of $(\param_{t})_{t\ge 0}$ converges to a stationary point of~\eqref{eq:expected_loss_minimization} almost surely. The rate of convergence for finding stationary points is $\tilde{O}(T^{-1/4})$ (measured using gradient mapping norm). (Thm. \ref{th: th_proj_sgd}, \ref{th: th_proj_sgd_smooth}, \ref{th: th_adagrad})
	    
	    \item The same result as above holds when $(\param_{t})_{t\ge 0}$ are generated by using stochastic heavy ball (see Alg. \ref{algorithm:shb}) and projected SGD with state-dependent Markov chain (see Thm. \ref{th: th_proj_sgd_state_dep}.
	\end{itemize}

	This is the same rate of convergence as in the i.i.d. case, up to $\log$-factors, which was obtained in \citep{davis2019stochastic} in terms of gradient mapping as shown in Thm.~\ref{th: th_proj_sgd_smooth}. Hence our analysis shows that the convergence of the algorithm and the order of the rate of convergence are not affected by such statistical bias in sampling training data, which was described earlier in this subsection. Furthermore, our result improves the rate of convergence of stochastic algorithms for constrained nonconvex expected loss minimization with dependent data \citep{lyu2022convergence}, see Thm.~\ref{th: th_proj_sgd_smooth} for the details. Moreover, we extend our analysis to obtain similar results for such projected SGD algorithms as adaptive gradient algorithm AdaGrad (see Algorithm \ref{algorithm:adagrad} and Theorem \ref{th: th_adagrad}) and SGD with heavy ball momentum (see Algorithm \ref{algorithm:shb} and Theorem \ref{th: th_shb}).

	
	\begin{table*}
		\centering
		{\small
			\begin{tabular}{lllll}\toprule
				& \begin{tabular}{@{}l@{}} $\min_{\param\in\mathbb{R}^p} f(\param)$ \\ $f\colon$ $L$-smooth\end{tabular} & \begin{tabular}{@{}l@{}} $\min_{\param\in\Param} f(\param)$\\ $f\colon$ $L$-smooth \end{tabular}  &   \begin{tabular}{@{}l@{}} Markovian \\ data \end{tabular}  & Constrained \\\midrule
				SMM~\citep{lyu2022convergence} & $\tilde{O}(\varepsilon^{-4})$ & $\tilde{O}(\varepsilon^{-8})^\dagger$ & \checkmark &\checkmark \\[1mm]
				\midrule
				SGD~\citep{sun2018markov,karimi2019non} & $\tilde{O}(\varepsilon^{-4})$  & $-$  & \checkmark &\xmark  \\[1mm]
				\midrule
				Proj. SGD \citep{davis2019stochastic} & $\tilde{O}(\varepsilon^{-4})$ & $\tilde{O}(\varepsilon^{-4})$  & \xmark & \checkmark  \\
				\midrule
				Proj. SGD-Sec.~\ref{sec: proj_sgd} & $\tilde{O}(\varepsilon^{-4})$ & $\tilde{O}(\varepsilon^{-4})$  & \checkmark &\checkmark  \\
				\midrule
				AdaGrad-Sec.~\ref{sec: adagrad} & $\tilde{O}(\varepsilon^{-4})$ & $\tilde{O}(\varepsilon^{-4})$ & \checkmark &\checkmark  \\
				\bottomrule
			\end{tabular}
		}
		\caption{{\small Complexity comparison for stochastic nonconvex optimization with non-i.i.d. data. Complexities in each column are the number of stochastic gradients to obtain: $\mathbb{E}\|\nabla f(\param)\|\leq\varepsilon$ and $\mathbb{E}\left[ \textup{dist}(0, \partial \varphi(\param)) \right]\leq\varepsilon$, respectively (where $\varphi$ is defined in~\eqref{eq:def_varphi_objective}). $^\dagger$This work showed the improved complexity $\tilde{O}(\varepsilon^{-4})$ under the additional assumption that the iterates of the algorithm are in the interior of $\Param$, which does not necessarily hold in the constrained case. We do not make such an assumption in this paper.}}
		\label{table: t1}
	\end{table*}

	\subsection{Notations}
	
	We fix $p\in \mathbb{N}$ to be the dimension of the ambient Euclidean space $\R^{p}$ equipped with the inner product $\langle \cdot,\cdot \rangle$ that also induces the Euclidean norm $\lVert \cdot \rVert$. For each $\eps>0$, let $B_{\eps}:=\{x\in \R^{p}\,|\, \lVert x \rVert \le \eps \}$ denote the $\eps$-ball centered at the origin. We also use the distance function defined as $\dist(\param, \Param) = \min_{\param'\in\Param} \| \param - \param' \|$ and the $\sigma$-algebra $\mathcal{F}_{t-k} =\sigma(\x_{1},\dots,\x_{t-k})$. We denote $f\colon\Param\rightarrow \R$ to be a generic objective function { for which we introduce the precise assumptions in Section~\ref{sec: grad_estim}}, where $\Param\subseteq \mathbb{R}^p$ is closed and convex. Let $\iota_{\Param}$ denote the indicator function of the set $\Param$, where $\iota_{\Param}(\param) = 0 \text{~if~} \param \in \Param$ and $ \iota_{\Param}(\theta) = +\infty \text{~if~} \param\not\in\Param$. Note that 
	\begin{align}\label{eq:def_varphi_objective}
		\argmin_{\param\in \Param} f(\param)   = \argmin_{\param\in \R^{p}} \left\{ \varphi(\param):=f(\param)+\iota_{\Param}(\param)\right\}. 
	\end{align}

	\subsection{Preliminaries on Stationarity Measures}\label{subsec: station_meas}
	{Since we do not expect the first-order optimality conditions to be satisfied exactly in a finite number of iterations in practice, we wish to estimate the worst-case number of iterations required to achieve an \emph{$\eps$-approximate} solution and the corresponding scaling with $\eps$.} To this end, we can relax the first-order optimality conditions as follows: For each $\eps>0$, we say $\param^{*}$ is an \textit{$\eps$-stationary point} (or $\eps$-approximate stationary point) for $f$ over $\Param$ if and only if $\textup{dist}(\mathbf{0}, \partial \varphi(\param^{*}) ) \le \eps$. We say a point $\param^{*}$ is \emph{approximately near stationary} for $f$ over $\Param$ if there exists some point $\hat{\param}$ near $\param$ that is approximately stationary for $f$ over $\Param$. We will make this notion precise through the following discussion.

	One of the central notions in the recent influential work by~\citet{davis2019stochastic} in analyzing convergence rates of first-order methods for constrained nonconvex problems  is the \emph{Moreau envelope}, which is a smooth approximation of an objective function that is closely related to proximal mapping. For a constant $\lambda>0$, we define the  \textit{Moreau envelope} $\varphi_{\lambda}$ of $\varphi$ defined in \eqref{eq:def_varphi_objective} as 
	\begin{align}
		\varphi_{\lambda}(\param):=\min_{\param'\in \R^{p}} \left( \varphi(\param') + \frac{1}{2\lambda} \lVert \param'-\param \rVert^{2}\right).\label{eq: def_moreu_env}
	\end{align}
	If $f$ is $\rho$-weakly convex and if $\lambda<\rho^{-1}$, then the minimum in the right hand side is uniquely achieved at a point $\hat{\param}$, which we call the \textit{proximal point} of $\param$. Accordingly, we define 
	the proximal map
	\begin{align}\label{eq:proximal_point}
		\hat{\param}&:=\prox_{\lambda\varphi}(\param) \notag \\
  &:= \argmin_{\param'\in \R^{p}} \Big(  \varphi(\param') + \frac{1}{2\lambda } \lVert \param'-\param \rVert^{2}\Big)
	\end{align}
	Also in this case, the Moreau envelope $\varphi_{\lambda}$ is $C^{1}$ with gradient given by (see~\citep{davis2019stochastic})
	\begin{align}\label{eq:moreau_grad}
		\nabla \varphi_{\lambda}(\param) = \lambda^{-1} (\param - \prox_{\lambda \varphi}(\param)). 
	\end{align}
	When $\param$ is a stationary point of $\varphi$, then its proximal point $\hat{\param}$ should agree with $\param$. Hence the gradient norm of the Moreau envelope $\varphi_{\lambda}$ may provide an alternative measure of stationarity. Indeed, as shown in~\cite{davis2019stochastic}, it provides a measure of \emph{near stationarity} in the sense that if $\lVert \nabla \varphi_{\lambda}(\param)\rVert$ is small, then since the proximal point $\hat{\param}$ in \eqref{eq:proximal_point} is within $\lambda \lVert \nabla \varphi_{\lambda}(\param)\rVert$ from $\param$, $\hat{\param}$ approximately stationary in terms of $\textup{dist}(\mathbf{0}, \partial \varphi(\hat{\param})) $:
	\begin{align}\label{eq:Moreau_near_stationary}
		\lVert \param - \hat{\param} \rVert \le \lambda \lVert \nabla \varphi_{\lambda}(\param) \rVert, ~~~~~ \varphi(\hat{\param}) \le \varphi(\param), \\ \textup{dist}(\mathbf{0},\, \partial \varphi(\hat{\param}) ) \le  \lVert \nabla \varphi_{\lambda}(\param) \rVert.
	\end{align}
	Note that the first and the last inequality above follows from the first-order optimality condition for $\hat{\param}$ together with \eqref{eq:moreau_grad} (see also Propositions \ref{prop:stationarty_gap_moreau} and \ref{prop:approx_stationarity} in Appendix \ref{sec: appendix}).

	Hence, in the literature of weakly convex optimization, it is common to state the results in terms of the norm of the gradient of Moreau envelope~\citep{davis2019stochastic,drusvyatskiy2019efficiency} which we will also adopt.
	When $g$ is additionally smooth, a commonly adopted measure to state convergence results is \emph{gradient mapping} which is defined as~\citep{nesterov2013gradient}
	\begin{align}
		\| \mathcal{G}_{1/\hat\rho}(\param_t)\| &= \hat\rho\left\| \param_t - \textup{proj}_{\Param}\left(\param_t - \frac{1}{\hat\rho} \nabla f(\param_t)\right) \right\|  \notag \\
		&=: \hat\rho \| \param_t - \tilde{\param}_t \|,\label{eq: gkr3}
	\end{align}
	for any $\lambda >0$, where we also defined $\tilde{\param}_t$.
	The results~\cite{davis2019stochastic}~\cite{drusvyatskiy2018error} showed how to translate the guarantees on the gradient of the Moreau envelope to gradient mapping by proving that
	\begin{align*}
		\|\mathcal{G}_{1/2\hat\rho}(\param)\| \leq \frac{3}{2} \| \nabla \varphi_{1/\hat\rho}(\param) \|.
	\end{align*}
	It is easy to show that a small gradient mapping implies that $\param_t$ is close to $\textup{proj}_{\Param}\left(\param_t - (1/\hat\rho) \nabla f(\param_t)\right)$ which itself is approximately stationary in view of Sec.~\ref{subsec: station_meas} which can be shown by using the definition of $\tilde{\param}_t$ and smoothness of $f$.
	Even though such an approximately stationary point can be computed in the deterministic case, computation of $\nabla f(\param_t)$ is not tractable in the stochastic case. However, as we show in Sec.~\ref{sec: smooth}, we can still output a point which is approximately stationary, in a tractable manner, with the claimed complexity results in our dependent data setting.

	\section{Stochastic Gradient Estimation}\label{sec: grad_estim}

	Denote as $\Delta_{[t-k,t]}$ the worst-case total variation distance between conditional distribution of $\x_{t}$ given $\x_{1},\dots,\x_{t-k}\in \Omega$ and the stationary distribution $\pi$. Namely, 
	\begin{align}
		\hspace{-0.3cm} \Delta_{[t-k,t]} := \sup_{\x_{1},\dots,\x_{t-k}} \lVert \pi_{t}(\cdot\,|\, \x_{1},\dots,\x_{t-k})  - \pi \rVert_{TV},
	\end{align}
	where $\pi_{t|t-k}=\pi_{t}(\cdot\,|\, \x_{1},\dots,\x_{t-k})$ denotes the probability distribution of $\x_{t}$ conditional on the past points $\x_{1},\dots,\x_{t-k}$.  

    Most of our theoretical results (except Theorem \ref{th: th_proj_sgd_state_dep} for state-dependent Markov chains, see Section \ref{sec: sjt4}) operate under the following three assumptions.

	 \begin{assumption}\label{assumption:A2}
	The function $f$ is $C^1$ smooth and has $\rho$-Lipschitz gradient and the set $\Param$ is closed and convex. There exists an open set $U$ containing $\Param$ and a mapping $G:U\times \Omega\rightarrow \R^{p}$ such that for all $\param\in \Param$, $ \E_{\x\sim \pi}\left[ G(\param, \x) \right] =\nabla  f(\param)$. Also $\param\mapsto G(\param, x)$ is $L_1$-Lipschitz for all $x$ for some $L_{1}>0$.
\end{assumption}

	\begin{assumption} \label{assumption:A1}
		We can sample a sequence of points  $(\x_{t})_{t\ge 1}$ in $\Omega$ in a way that: (1) For each $x\ge 0$,  $\Delta_{[t,t+N]}$ is non-increasing in $N\ge 0$; and (2) $\lim_{N\rightarrow\infty} \Delta_{[t, t+N]}=0$ for all $t\ge 0$; and (3) there exists a sequence $k_{t}\in [0,t]$, $t\ge 1$ such that  $\Delta_{[t-k_{t}, t]}\rightarrow 0$ and  $\sum_{t=1}^{\infty} \alpha_{t}\Delta_{[t-k_{t},t]} <\infty$, where $\alpha_{t}>0$ denotes the stepsize in the first-order method. 
	\end{assumption}



\begin{assumption}\label{assumption:A3}
    Assume either of the two: \emph{(i)} There is $L\in (0,\infty)$ such that for each $t\ge 1$ and $\param\in \Param$, $ \E\left[ \lVert G(\param, \x_{t+1})  \rVert \,|\,  \mathcal{F}_{t}\right]  \le L$ and the process $(\x_{t})_{t\ge 0}$ is a function of some time-homogeneous Markov chain; or \emph{(ii)} There  is $L\in (0,\infty)$ such that $\lVert G(\param, \x) \rVert \le L$ for all $\param,x$. 
\end{assumption}

\cref{assumption:A2} is about smoothness of the objective and stochastic gradient operator $G$. The former is standard in the literature of stochastic constrained first-order methods and the latter is also common when we additionally work with dependent data(see, e.g., \cite{davis2019stochastic, sun2020markov, lyu2022convergence}).


\cref{assumption:A1} states that: (1) The $N$-step conditional distribution $\pi_{t+N|t}$ can only be closer to the stationary distribution $\pi$ when $N$ increases; (2) the $N$-step conditional data distribution $\pi_{t+N|t}$ converges to the stationary distribution $\pi$ asymptotically; and (3) such convergence (mixing) occurs at a sufficiently fast rate. The sequence $k_{t}$ plays a critical role in controlling dependence in data samples. The key idea is that, when analyzing quantities at time $t+1$, one conditions on a `distant past' $t-k_{t}$ (instead of the present $t$) and approximates the multi-step conditional data distribution $\pi_{t+1|t-k_{t}}$ by the stationary distribution $\pi$. The error of such approximation in the total variation distance is bounded by $\Delta_{[t-k_{t}, t]}$. \cref{assumption:A1} requires that this quantity should be summable after being multiplied by the stepsize $\alpha_{t}$. 

There are two notable special cases that satisfy \cref{assumption:A1}. First, \cref{assumption:A1} is trivially satisfied (with $k_{t}\equiv 0$) in the i.i.d. case since then $\pi_{s|t} \equiv \pi$ whenever $s>t$. 
 
Second, suppose $\x_{t}$ is given by a function $g$ of some underlying time-homogeneous Markov chain $X_{t}$ with a stationary distribution $\pi$. In this case \cref{assumption:A1}(1) holds by Scheff\'{e}'s lemma (see, e.g., Lemma 2.1 in  \cite{tsybakov2004introduction}). (Here time-inhomogeneity is not necessary.) If $X_{t}$ is irreducible and aperiodic on a finite state space, then \cref{assumption:A1}(2) holds with $\Delta_{[t-k,t]}=O(\exp(-ck))$ for some constant $c>0$ independent of $t$ \citep{levin2017markov}. So Assumption~\ref{assumption:A1}(3) is verified for any $k_{t} \ge C\log t$ for $C>0$ large enough so that $\sum_{t\ge 1} \exp(-ck_{t})<\infty$ and for any $\alpha_{t}=O(1)$.
In the case when the underlying Markov chain $X_{t}$ has countably infinite or uncountable state space, then a more general condition for geometric ergodicity is enough to imply Assumption~\ref{assumption:A1} (see, e.g., \citep{levin2017markov, meyn2012markov}). 
	See~\cite{lyu2020online} and~\citep{sun2018markov} for concrete applications and sampling methods that satisfy this assumption.
	This assumption is common in the literature~\citep{bhandari2018finite,lyu2022convergence,lyu2020online,sun2018markov,nagaraj2020least} and i.i.d. sampling is another special case.

We emphasize that Assumption \ref{assumption:A1} does not necessarily reduce to time-homogeneous and state-independent Markov chains.	Our main focus is using Assumption~\ref{assumption:A1} which is the main assumption on the data in most of the works we compare with. However, we also discuss another popular setting of modeling dependent data samples by \textit{state-dependent Markov chain}. See \ref{assumption:A1'}-\ref{assumption:A2'} and Thm. \ref{th: th_proj_sgd_state_dep}. 
 
Next, we discuss \cref{assumption:A3} on boundedness of stochastic gradients. In the i.i.d. case, it is standard to assume uniform boundedness of $\E_{\x\sim \pi}[\lVert G(\param,\x)\rVert]$ for each $\param\in \Param$ \citep{davis2019stochastic, davis2020stochastic}. In the non-i.i.d. case, it has been customary to make stronger assumption of uniform boundedness of $G(\param,\x)$ even in the unconstrained nonconvex case \cite{sun2018markov, dorfman2022adapting}, which does not properly generalize the standard assumption in the i.i.d. case. This is mostly for controlling the error of multi-step conditional expectation of the stochastic gradient by its stationary expectation, which is the crucial issue in the non-i.i.d. case that is non-existent in the i.i.d. case. 

In this work, we are able to analyze the non-i.i.d. setting under a much weaker condition in Assumption~\ref{assumption:A3}(i) that only assumes one-step conditional expectation of the norm of the stochastic gradient is bounded. Although for a technical reason we will also need to assume that the data samples $(\x_{t})_{t\ge 0}$ are given as a function of some time-homogeneous Markov chain, Assumption~\ref{assumption:A3}(i) properly generalizes the standard assumptions in the i.i.d. case. In addition, We also analyze non-i.i.d. setting under uniformly bounded stochastic gradients but with more general data sampling setting (\cref{assumption:A3}(ii)), including time-inhomogeneious and non-Markovian setting.

	Now we state a key lemma that handles the bias due to dependent data and is algorithm independent. In the sequel, we will invoke this lemma for different algorithms such as SGD, AdaGrad or SGD with heavy ball momentum.
	
	\begin{lemma}[Key lemma]\label{lem:TV_linear_change_bound}
		Let Assumptions~\ref{assumption:A2}, ~\ref{assumption:A1},~\ref{assumption:A3} hold and $\param_{t}$ be generated according to Algorithm~\ref{algorithm:prox_grad},~\ref{algorithm:adagrad} or ~\ref{algorithm:shb}.
		Fix $\hat{\rho}>\rho$ and denote $\hat{\param}=\prox_{\varphi/\hat{\rho}}(\param)$ and fix $1\le k \le t$.  Then 
		\begin{align}
			\Big| \E\big[  \langle &\hat{\param}_{t} - \param_{t},\, G(\param_{t}, \x_{t+1}) \rangle \,|\, \mathcal{F}_{t-k}\big] \\
			&\hspace{-0.3cm}-\langle \hat{\param}_{t} - \param_{t},\, \E_{\x \sim \pi}\left[ G(\param_{t}, \x) \right] \rangle \Big|   \nonumber \le  \frac{ 4 L^2}{\hat\rho-\rho} \,  \Delta_{[t-k,t]}   \nonumber
			\\
			&\hspace{-0.3cm}+ \frac{2L(L_1+\hat\rho )}{\hat\rho-\rho} \,  \E\left[ \sum_{s=t-k}^{t-1} \alpha_{s}\lVert G(\param_{s},\x_{s+1}) \rVert \,\bigg|\, \mathcal{F}_{t-k} \right]. \nonumber
		\end{align}
	\end{lemma}
	
	\vspace{-0.4cm}
	This lemma borrows some ideas from~\citep{lyu2022convergence}.
	The important difference is that, the result of the lemma makes it explicit the dependence on the step size and gradient norms to be applicable with AdaGrad. This is needed because the step size of AdaGrad does not have a specific decay schedule. The proof is given in Section~\ref{sec: app_proofs}.
	
	\section{Convergence Rate Analysis}\label{sec: sec3}
	\subsection{Projected SGD with Dependent Data}\label{sec: proj_sgd}

	Now we state our first main result in this work, which extends the convergence result of projected SGD with i.i.d. samples in \citep{davis2019stochastic} to the general dependent sample setting.
	This result improves the existing complexity of stochastic algorithms from~\citep{lyu2022convergence} for solving constrained nonconvex stochastic optimization under dependent data, see Section~\ref{sec: smooth} for details.  We use the notion of \emph{global convergence} with respect to arbitrary initialization below. The proof of this result is in Appendix~\ref{app: psgd}.

	\begin{algorithm}[H]
		\small
		\caption{Projected Stochastic Gradient Algorithm (PSGD)}
		\label{algorithm:prox_grad}
		\begin{algorithmic}[1]
			\STATE \textbf{Input:} Initialize $\param_{1}\in \Param \subseteq \R^{p}$; $T > 0$;\,  Stepsizes $(\alpha_{t})_{t\ge 1}$
			\STATE Sample $\tau$ from  $\{1,\dots,T\}$ independently of everything else where $\P(\tau=k)=\frac{\alpha_k}{\sum_{t=1}^T \alpha_t}$. 
			\STATE \textbf{For $t=1,2,\dots,T$ do:}
			\STATE \qquad Sample $\x_{t+1}$ from $\pi_{t+1}=\pi_{t+1}(\cdot\,|\, \x_{1},\dots,\x_{t})$
			\STATE \qquad $\param_{t+1} \leftarrow \proj_{\Param} \left(  \param_{t} - \alpha_{t} G(\param_{t}, \x_{t+1}) \right)$\label{eq: woe3}
			\STATE \textbf{End for}
			\STATE \textbf{Return:} $\param_{T}$ (Optionally, $\param_{T}^{\text{out}}$ as either $\param_{\tau}$ or $\arg\min_{\param\in\{\param_1, \dots, \param_T \}} \| \nabla \varphi_{1/\hat{\rho}}(\param) \|^2$.) 
		\end{algorithmic}
	\end{algorithm}

\begin{theorem}[Projected stochastic gradient method]\label{th: th_proj_sgd}
		Let Assumptions \ref{assumption:A2}-\ref{assumption:A3} hold and $(\param_{t})_{t\ge 1}$ be a sequence generated by Algorithm \ref{algorithm:prox_grad}. 
		Fix $\hat{\rho}>\rho$. Then the following hold:
		\begin{description}
			\item[(i)] (Rate of convergence)  For each $T\ge 1$, 
			\begin{multline}
				\E\left[     \lVert \nabla \varphi_{1/\hat{\rho}}(\param_{T}^{\mathrm{out}}) \rVert^{2} \right]  = O\Big(\frac{1}{\sum_{k=1}^{T} \alpha_{k}}\Big(\sum_{t=1}^T \alpha_t^2 \\
				+ \sum_{t=1}^T k_t\alpha_t \alpha_{t-k_t} + \sum_{t=1}^T \alpha_t \mathbb{E}[\Delta_{[t- k_t, t]}] \Big)\Big).
			\end{multline}
			In particular, with $\alpha_t = \frac{c}{\sqrt{t}}$ for some $c> 0$ and under exponential mixing, we have that $ \E\left[     \lVert \nabla \varphi_{1/\hat{\rho}}(\param_{T}^{\mathrm{out}}) \rVert \right] \leq \varepsilon $ with $\tilde{O}\left( \varepsilon^{-4} \right)$ samples. 
			
			\item[(ii)] (Global convergence) Further assume that $\sum_{t=0}^{\infty} k_{t}\alpha_{t} \alpha_{t-k_{t}}<\infty$. Then $\lVert \nabla \varphi_{1/\hat{\rho}}(\hat{\param}_{t}) \rVert \rightarrow 0$ as $t\rightarrow \infty$ almost surely. Furthermore, 
   $\param_{t}$ converges to the set of all stationary points of $f$ over $\Param$. 
		\end{description}
	\end{theorem}
	
	If $(\x_{t})_{t\ge 1}$ is exponentially mixing, then Theorem \ref{th: th_proj_sgd}\textbf{(ii)} holds with $\alpha_{t}=t^{-1/2} (\log t)^{-1-\eps}  $ for any fixed $\eps>0$ and  $k_{t}=O(\log t)$. 

	\subsection{AdaGrad with Dependent Data}\label{sec: adagrad}
	We next establish the convergence of AdaGrad with dependent data and constrained nonconvex optimization. We will use AdaGrad with scalar step sizes (see Alg. \ref{algorithm:adagrad}), which is also referred to as AdaGrad-norm~\citep{ward2019adagrad,levy2017online,streeter2010less}.
    
    \begin{algorithm}[H]
		\small
		\caption{AdaGrad-norm~\citep{streeter2010less}}
		\label{algorithm:adagrad}
		\begin{algorithmic}[1]
			\STATE \textbf{Input:} Initialize $\param_{1}\in \Param \subseteq \R^{p}$; $T > 0$;\,   $(\alpha_{t})_{t\ge 1}$ ;\, $v_{0}>0$;\, $\alpha > 0$
			\STATE Optionally, sample $\tau$ from  $\{1,\dots,T\}$ independently of everything else where $\P(\tau=k)=\frac{1}{T}$. 
			\STATE \textbf{For $t=1,2,\dots,T$ do:}
			\STATE \qquad Sample $\x_{t+1}$ from $\pi_{t+1}=\pi_{t+1}(\cdot\,|\, \x_{1},\dots,\x_{t})$
			\STATE \qquad $v_t = v_{t-1} + \| G(\param_t,\x_{t+1}) \|^2$
			\STATE \qquad $\alpha_t = \frac{\alpha}{\sqrt{v_t}}$
			\STATE \qquad $\param_{t+1} \leftarrow \proj_{\Param} \left(  \param_{t} - \alpha_t G(\param_{t}, \x_{t+1}) \right)$
			\STATE \textbf{End for}
			\STATE \textbf{Return:} $\param_{T}$ (Optionally, $\param_{T}^{\text{out}}$ as either $\param_{\tau}$ or $\arg\min_{\param\in\{\param_1, \dots, \param_T \}} \| \nabla \varphi_{1/\hat{\rho}}(\param) \|^2$.)
		\end{algorithmic}
	\end{algorithm}
	
	For this section, we introduce an additional assumption on the boundedness of the objective values.
	\begin{assumption}{(A4)} \label{assumption:A4}
		There exists $C_\varphi\in (0,\infty)$ such that $| f(\param) | \le C_\varphi$ for all $\param\in \Param$.
	\end{assumption}
	Compared to projected SGD, the step size of AdaGrad does not have a specific decay schedule, which makes it challenging to use the existing bias analyses for dependent data (for example the idea from~\citep{lyu2022convergence}) since they critically rely on knowing the precise decay rate of the step sizes.

	To be able to apply such an analysis for adaptive algorithms, we use a generalized result in Lem.~\ref{lem:TV_linear_change_bound} and use the particular form of AdaGrad step size in Thm.~\ref{th: th_adagrad} to achieve the optimal $\tilde{O}(\varepsilon^{-4})$ complexity. Full proof of the result is given in Appendix~\ref{app: adagrad_appendix}.

		\begin{theorem}[AdaGrad-norm]\label{th: th_adagrad}
		Let \cref{assumption:A2}-\ref{assumption:A3} and \cref{assumption:A4} hold and $(\param_{t})_{t\ge 1}$ be a sequence generated by Algorithm \ref{algorithm:adagrad}. 
		Fix $\hat{\rho}>\rho$ and a nondecreasing, diverging sequence $(k_{t})_{t\ge 1}$. 
		Then,
		for each $T\ge 1$, 
		\begin{align}
			&\E\left[     \lVert \nabla \varphi_{1/\hat{\rho}}(\param_{T}^{\mathrm{out}}) \rVert^{2} \right]   \\
			&\qquad = O\bigg(\frac{k_T \log(TL^2)}{\sqrt{T}} +  \frac{1}{T}\sum_{t=1}^{T} \mathbb{E}[\Delta_{[t-k_{t}, t]}]\bigg). \nonumber 
		\end{align}
	\end{theorem}
	
		We note that unlike Thm.~\ref{th: th_proj_sgd}, for AdaGrad we only prove nonasymptotic complexity results and not asymptotic convergence statements for the output sequence of the algorithms. 
		Even though asymptotic convergence of AdaGrad with i.i.d. data is proven in~\citep{li2019convergence}, the technique in that paper relies on using the inequality~\eqref{eq: roit4} multiplied with $\alpha_t$. However, the specific form of~\eqref{eq: roit4} is important in our development to use Lem.~\ref{lem:TV_linear_change_bound} to handle the dependent data, since $\alpha_t$ brings additional stochastic dependencies.
		Even though we believe an appropriate modification of Lem.~\ref{lem:TV_linear_change_bound} can be possible, we do not pursue such generalization in the present work. 

		Since the step size in this case is nonincreasing,~\cref{assumption:A1} reduces to $\sum_{t=1}^\infty \Delta_{[t-k, t]} < \infty$. This, for example, is satisfied for the exponential mixing case that is mentioned in Theorem~\ref{th: th_proj_sgd} and considered in the previous work~\citep{lyu2022convergence,sun2018markov,bhandari2018finite}.

	\subsection{Stochastic Heavy Ball with Dependent Data}
	Because of space limitations, we defer the formal description of SGD with heavy ball momentum to the appendix (Algorithm~\ref{algorithm:shb}) and include a summary of the complexity result here. The extended theorem for this case, including the asymptotic convergence of the sequence and the proofs are given in Appendix~\ref{sec: stoc_hb}.
		\begin{theorem}\label{th: th_shb}
		Let \cref{assumption:A2}-\ref{assumption:A3} hold and $(\param_{t})_{t\ge 1}$ be a sequence generated by Algorithm \ref{algorithm:shb}. 
		Fix $\hat{\rho}\geq 2\rho$. 
		Then, for any momentum parameter $\beta \in (0, 1]$ and $T\ge 1$:
		\begin{multline}
			\E\left[     \lVert \nabla \varphi_{1/\hat{\rho}}(\param_{T}^{\mathrm{out}}) \rVert^{2} \right]  = O\bigg(\frac{1}{\beta^2\sum_{k=1}^{T} \alpha_{k}}\bigg(\sum_{t=1}^T \alpha_t^2 \\
			+ \sum_{t=1}^T k_t\alpha_t \alpha_{t-k_t} + \sum_{t=1}^T \alpha_t \mathbb{E}[\Delta_{[t- k_t, t]}] \bigg)\bigg).
		\end{multline}
			
	\end{theorem}
	
\vspace{-0.3cm}
Our analysis for the heavy ball method appears to be more flexible compared to~\citep{mai2020convergence} even when restricted to the i.i.d. case. In this case, we allow variable step sizes $\alpha_t=\frac{\gamma}{\sqrt{t}}$ whereas~\citep{mai2020convergence} requires constant step size $\alpha_t=\alpha = \frac{\gamma}{\sqrt{T}}$. We can also use any momentum parameter $\beta\in(0, 1]$ whereas~\citep{mai2020convergence} restricts to $\beta=\alpha$.  This point is important since in practice $\beta$ is used as a tuning parameter.

	\subsection{Proximal SGD with Dependent Data}
	In this section, we describe how our developments for stochastic gradient method extends to the proximal case, using the ideas from~\citep{davis2019stochastic}.
	In particular, the problem we solve in this section is
	\begin{align}\label{eq: sd4}
		\param^{*} \in \argmin_{\param\in \R^{p}} \Big( \varphi(\param)&:= f(\param) + r(\param)\Big),
	\end{align} 
	where  $f$ is as in \eqref{eq:expected_loss_minimization} and $r\colon \mathbb{R}^p \to \mathbb{R} \cup {\{ +\infty \}}$ is a convex, proper, closed function.
	In this case, in step~\ref{eq: woe3} of Algorithm \ref{algorithm:prox_grad}, we use $\prox_{\alpha_t r}$ instead of $\proj_{\Param}$ to define $\param_{t+1}$.
	We include the following result combining the ideas from Lem.~\ref{lem:TV_linear_change_bound}, Thm.~\ref{th: th_proj_sgd} and \cite{davis2019stochastic} for proving convergence of proximal stochastic gradient algorithm with dependent data.
	Full details are given in Appendix~\ref{sec: prox_extension}.
	\begin{theorem}\label{th:prox}
		Let \cref{assumption:A2}-\ref{assumption:A3} hold, $r$ be convex, proper, closed and $(\param_{t})_{t\ge 1}$ be a sequence generated by Algorithm \ref{algorithm:prox_grad} where we use $\prox_{\alpha_t r}$ instead of $\proj_{\Param}$ in step~\ref{eq: woe3}.
		Fix $\hat{\rho}>\rho$. For each $T\ge 1$, 
		\begin{multline}
			\E\left[     \lVert \nabla \varphi_{1/\hat{\rho}}(\param_{T}^{\mathrm{out}}) \rVert^{2} \right]  = O\bigg(\frac{1}{\sum_{k=1}^{T} \alpha_{k}}\bigg(\sum_{t=1}^T \alpha_t^2 \\
			+ \sum_{t=1}^T k_t\alpha_t \alpha_{t-k_t} + \sum_{t=1}^T \alpha_t \mathbb{E}[\Delta_{[t- k_t, t+1]}] \bigg)\bigg).
		\end{multline}
	\end{theorem}

\subsection{Projected SGD with state-dependent Markovian data}\label{sec: sjt4}
Next, we state an analogous result to Theorem \ref{th: th_proj_sgd} when the data samples $(\x_{t})_{t\ge 0}$ form a \textit{state-dependent Markov chain}. It extends the corresponding results in~\citep{karimi2019non,tadic2017asymptotic} to the constrained case. One difference is that in the constrained case, we need a slightly stronger assumption on the norms of the gradients, see~\ref{assumption:A3}.
The assumptions below were adapted from ~\citep{karimi2019non} and~\citep{tadic2017asymptotic}.
\begin{assumption}\label{assumption:A1'}
The sequence of data samples $(\x_{t})_{t\ge 0}$ form a \textit{state-dependent Markov chain controlled by $\param\in \Param$}, denoted as $(X_t)_{t\ge 0}$.  That is, for each $\param\in \Param$, there exists a Markov kernel $P_{\param}:\Omega\rightarrow \Omega$ such that for any bounded measurable function $H$, 
	\begin{equation}\label{eq: gk4}
		\mathbb{E}[H(X_{t+1}) | \mathcal{F}_t] = P_{\param_{t}} H(X_t),
	\end{equation}
	where $\mathcal{F}_{t}:=\sigma( X_{0}, \param_{0}, X_{1},\param_{1},\dots, X_{t},\param_{t})$. 
\end{assumption}

\begin{assumption}\label{assumption:A2'}
	There is a Lipschitz continuous solution to the Poisson equation for $(X_{t})_{t\ge 0}$. That is, there exists a measurable function $\hat G$ such that for each $\param\in \Param, x\in \Omega$, 
	\begin{equation}\label{eq: gk5}
		\hat G(\param, x) - P_{\param}\hat G(\param, x) = G(\param, x) - \nabla f(\param),
	\end{equation}
	where $f$ denotes the objective function in \eqref{eq:expected_loss_minimization} and $G(\param, x)$ is as in \cref{assumption:A2}. Furthermore, There exists $C_1, C_2, C_3$ such that 
	\begin{align}\label{eq: gk6}
		&\|\hat G(\param, x) \| \leq C_1, ~~~ \| P_{\param} \hat G(\param, x) \| \leq C_2, \\
		&\sup_{x} \| P_{\param}\hat G(\param, x) - P_{\param'}\hat G(\param', x)\| \leq C_3 \| \param - \param'\|.
	\end{align}
\end{assumption}

\begin{theorem}[Projected SGD with state-dependent MC data]\label{th: th_proj_sgd_state_dep}
	Let Assumptions~\ref{assumption:A2},~\ref{assumption:A3},~\ref{assumption:A1'},~\ref{assumption:A2'} hold and $(\param_{t})_{t\ge 1}$ be a sequence generated by Algorithm \ref{algorithm:prox_grad}. A complexity result as in Theorem \ref{th: th_proj_sgd} \textbf{(i)}  still hold with possibly different constants. See Theorem \ref{thm:PSGD_state_dependent_MC} for details.
\end{theorem}

While Lemma \ref{lem:TV_linear_change_bound} was the key to establish convergence of PSGD (Theorem \ref{th: th_proj_sgd}) under the mixing condition in \cref{assumption:A1}, a similar role is played by the solution of Poisson equation stated in \cref{assumption:A2'} for the state-dependent case. The proof of Theorem \ref{th: th_proj_sgd_state_dep} follows the same lines as Theorem~\ref{th: th_proj_sgd} using a similar analysis as in  \citep{karimi2019non} for the bias and properties of the sequences $\hat\param_t, \param_t$. See Appendix \ref{sec: we3}. 
	
	\subsection{Complexity for Constrained Smooth  Optimization with Dependent Data}\label{sec: smooth}
	We next compare our complexity with the one derived in~\citep{lyu2022convergence} for constrained smooth nonconvex optimization with dependent data which, to our knowledge, is the only complexity result for this setting.
	First, we introduce the next assumption to replace~\cref{assumption:A2}.
	We next show how to translate our result to a direct stationarity measure in view of Sec.~\ref{subsec: direct_stat} to compare with the $\tilde O(\varepsilon^{-8})$ complexity result in~\citep{lyu2022convergence} for an equivalent stationarity measure (see Sec.~\ref{subsec: direct_stat} for details). The proof of the result is given in Appendix~\ref{app: smooth}.
	\begin{theorem}[Sample complexity]\label{th: th_proj_sgd_smooth}
		Let \cref{assumption:A2}-\ref{assumption:A3} hold and $(\param_{t})_{t\ge 1}$ be a sequence generated by any of the Algorithms \ref{algorithm:prox_grad}, \ref{algorithm:adagrad}, and \ref{algorithm:shb}.
		Fix $\hat{\rho}>\rho$, assume $\Delta_{[t- k_t, t+1]} = O(\lambda^{k_t})$ for $\lambda < 1$ and that $\Param$  is compact. Pick $\hat t$ randomly from $\{1,\dots,T\}$ as in the respective theorems for the algoritms, let $\hat N=O(\varepsilon^{-2})$, and define 
		\begin{equation}
			\breve{\param}_{\hat t+1} = \proj_{\Param}\bigg(\param_{\hat t} - \frac{1}{\hat N}\sum_{i=1}^{\hat N} \nabla \ell(\param_{\hat t}, \x^{(i)})) \bigg). 
		\end{equation}
		Then $\mathbb{E}\left[\mathrm{dist}(\mathbf{0}, \partial(f+\iota_{\Param})(\breve{\param}_{\hat t+1}))\right] \leq \varepsilon$ with $\tilde{O}(\varepsilon^{-4})$ samples. 
	\end{theorem}
	The assumption on $\Delta_{[t- k_t, t+1]}$ and hence the dependent sampling is consistent with the related works~\citep{lyu2022convergence,sun2018markov}.

	Even though gradient of the Moreau envelope is a \emph{near approximate stationarity} measure, in the specific case of this section, we show that we can output a point that is \emph{approximately stationary} with respect to the direct stationarity measure in Prop.~\ref{prop:approx_stationarity}(i) (also mentioned in Section~\ref{subsec: station_meas}).
	This permits a direct comparison with the previous result on constrained nonconvex optimization with dependent data~\citep{lyu2022convergence} and shows our improvement.
	Lemma~\ref{lem: post_process_output} in the appendix gives the necessary post-processing step for this, which is used in Theorem~\ref{th: th_proj_sgd_smooth}.

	\section{Application: Online Dictionary Learning}
	\label{section:ODL}
	Consider the \textit{online dictionary learning} (ODL) problem, which is stated as the stochastic program 
	\begin{equation}\label{eq:ODL_f}
		\begin{aligned}
			\min&_{\param \in \Param\subseteq \R^{p\times r} } \Big( f(\param):=\E_{\X\sim\pi}\left[ \ell(\X,\param) \right]\Big) \text{~where} \\
			&\ell(\X,\param) :=\inf_{H\in \Param' \subseteq \R^{p\times n}} d\left(\X  , \param H\right) + R(H)
		\end{aligned}
	\end{equation}
	where $d(\cdot,\cdot):\R^{p\times n}\times \R^{p\times n}\rightarrow [0,\infty)$ is a multi-convex function that measures dissimilarity between two $p\times n$ matrices (e.g., the squared frobenius norm, KL-divergence), $R:\R^{p\times n}\rightarrow [0,\infty)$ denotes a convex regularizer for the code matrix $H$, and $r$ is an integer parameter for the rank of the intended compressed representation of data matrix $\X$. In words, we seek to learn a single dictionary matrix $\param\in \R^{p\times r}$ within the constraint set $\Param$ (e.g., nonnegative matrices with bounded norm), which provides the best linear reconstruction (w.r.t. the $d$-metric) of an unknown random matrix $X$ drawn from some distribution $\pi$. Here, we may put $L_{1}$-regularization on $H$ in order to promote dictionary $\param$ that enable sparse representation of observed data. 
	
	The most extensively investigated instance of the above ODL problem is when $d$ equals the squared Frobenius distance. In this case, Mairal et al. \citep{mairal2010online} provided an online algorithm based on the framework of stochastic majorization-minimization \citep{mairal2013stochastic}. 
	A well-known result in \citep{mairal2010online} states that the above algorithm converges almost surely to the set of stationary points of the expected loss function $f$ in \eqref{eq:ODL_f}, provided the data matrices $(\X_{t})_{t\ge 1}$ are i.i.d. according to the stationary distribution $\pi$. Later Lyu, Needell, and Balzano \citep{lyu2020online} generalized the analysis to the case where $(\X_{t})$ are given by a function of some underlying Markov chain. Recently, \citet{lyu2022convergence} provided the first convergence rate bound of the ODL algorithm in \citet{mairal2010online} of order $O((\log t)^{1+\eps}/t^{1/4})$ for the empirical loss function and $O((\log t)^{1+\eps}/t^{1/8})$ for the expected loss function for arbitrary $\eps>0$.

 	\begin{figure*}[ht!]
		\centering
		\includegraphics[width=1\linewidth]{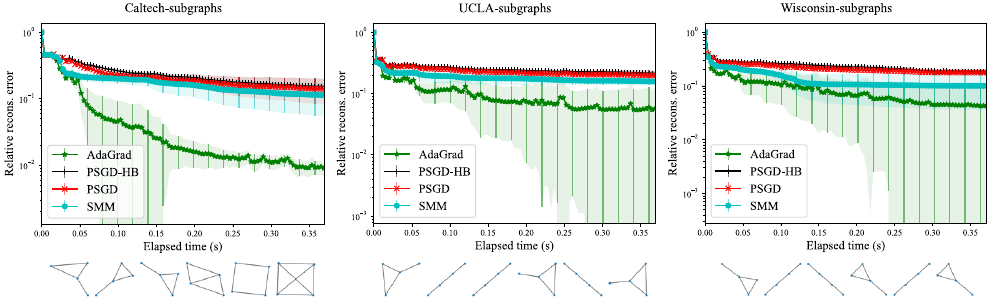}
		\vspace{-0.5cm}
		\caption{Plot of reconstruction error vs. elapsed time for four algorithms for online NMF: AdaGrad, PSGD-Heavy Ball, PSGD, and SMM. Data stream is a sequence of 4-node subgraph adjacency matrices sampled by an MCMC motif-sampling algorithm in \cite{lyu2023sampling} from three college Facebook networks \cite{traud2012social}. Six consecutive Markovian samples of subgraphs are shown in each plot. Shaded region represents one standard deviation from ten runs.}
		\label{fig:exp}
	\end{figure*}
	
	Suppose we are given a sequence of data matrices $(\X_{t})_{t\ge 1}$ that follows $\pi$ in some asymptotic sense. Under some mild assumptions, one can compute the subgradient of the loss function $\param\mapsto \ell(\X_{t},\param)$ in two steps and can perform a standard stochastic projected gradient descent: 
	\begin{align}\label{eq:ODL_PSGD}
		\begin{cases}
			&H_{t} \leftarrow \argmin_{H\in \Param'} d\left(\X  , \param_{t-1} H\right) +\lambda \lVert H \rVert_{1},  \\
			&G(\param_{t-1},\X_{t}) =\nabla_{\param} d(\X_{t},\param_{t-1} H_{t}), \\
			& \param_{t}\leftarrow \textup{Proj}_{\Param}\left(  \param_{t-1} - \alpha_{t}  G(\param_{t-1},\X_{t}) \right).
		\end{cases}
	\end{align}
	For instance, consider the following standard assumption on `uniqueness of sparse coding problem':

	\begin{assumption} \label{assumption:A5}
		For each $\X$ and $\param$, $\inf_{H\in \Param'\subseteq \R^{p\times n}} d(\X, \param H) + R(H)$ admits a unique solution in $\Param'\subseteq \R^{p\times n}$. 
	\end{assumption}
	\noindent Note that \cref{assumption:A5} is trivially satisfied if $R(H)$ contains a regularization $\kappa_{2} \lVert H \rVert^{2}_{F}$ for some $\kappa_{2}>2$. Under \cref{assumption:A5}, Danskin's theorem \citep{bertsekas1997nonlinear} implies that the function $\param\mapsto \ell(\X,\param)$ is differentiable and satisfies $\nabla_{\param} \ell(\X,\param)= \nabla_{\param} d(\X, \param \H^{\star})$, where $\H^{\star}$ is the unique solution of $\inf_{H\in \Param'\subseteq \R^{p\times n}} d(\X, \param H) + \lambda \lVert H\rVert_{1}$. Hence we may choose $G(\param_{t-1},\X_{t})=\nabla_{\param} d(\X_{t}, \param \H_{t})$ in \eqref{eq:ODL_PSGD}.

	Notice that \eqref{eq:ODL_PSGD} is a projected SGD algorithm for the ODL problem \eqref{eq:ODL_f}, which is a constrained nonconvex problem.  Zhao et al. \citep{zhao2017online} provided asymptotic analysis of this algorithm  (especially for online nonnegative matrix factorization) for general dissimilarity metric $d$. For a wide class of dissimilarity metrics such as  Csiz\'{a}r $f$-divergence, Bregman divergence, $\ell_{1}$ and $\ell_{2}$ metrics, and Huber loss, this work showed that when the data matrices are i.i.d. and the stepsizes $\alpha_{t}$ are non-summable $(\sum_{t=1}^{\infty} \alpha_{t}=\infty)$ and square-summable $(\sum_{t=1}^{\infty} \alpha_{t}^{2}<\infty)$, then the sequence of dictionary matrices $(\param_{t})_{t\ge 1}$ obtained by \eqref{eq:ODL_PSGD}, regardless of initialization, converges almost surely to the set of stationary points of \eqref{eq:ODL_f}. The asymptotic analysis uses a rather involved technique inspired from dynamical systems literature and does not provide a rate of convergence. Moreover, such asymptotic guarantees has not been available to the more general Markovian data setting. 

	When the function $\param\mapsto \ell(\X,\param)$ for each $\X$ is $\rho$-weakly convex for some $\rho>0$, then the expected loss function  in \eqref{eq:ODL_f} is also $\rho$-weakly convex, so in this case a direct application of the main result in \citep{davis2020stochastic} would yield a rate of convergence $O((\log t)/t^{1/4})$ for \eqref{eq:ODL_PSGD} with i.i.d. data matrices $X_{t}$. Such hypothesis of weak convexity of the loss function is implied under smoothness of $d$ (Assumption \ref{assumption:A6}). Then our main results extends the theoretical guarantees for \eqref{eq:ODL_PSGD} to more general setting when $(\X_{t})$ are given as a function of some underlying Markov chain with exponential mixing, and also extends to other variants of PSGD such as the AdaGrad (Algorithm \ref{algorithm:adagrad}) and the stochastic heavy ball (Algorithm \ref{algorithm:shb}). The full statement of this result for ODL with stochastic first-order methods on non-i.i.d. data is stated in \cref{cor: odl} in Appendix \ref{sec:ODL_app}. \textit{To our best knowledge, this is the first time that projected SGD with adaptive step sizes has been applied to ODL problems with optimal complexity bounds for the general Markovian data case.} 

	\paragraph{Numerical validation.}

	We now illustrate the empirical performance of Alg.~\ref{algorithm:prox_grad},~\ref{algorithm:adagrad},~\ref{algorithm:shb} to verify our theoretical findings for ODL with dependent data. However, we highlight that the main contribution of our paper is \emph{theoretical}: obtaining the optimal complexity for constrained nonconvex problems with dependent data.
	Moreover, the algorithms we analyze, namely, projected SGD, SGD with momentum and AdaGrad are the default solvers in most of the libraries for machine learning/deep learning such as PyTorch/TensorFlow and their empirical success is well-established.
	Hence, the results here are not meant to be complete benchmarks, but rather empirical support for our theory.

	For generating the samples, we used networks for 3 different schools (\texttt{Caltech36, UCLA26, Wisconsin87}) from the Facebook100 dataset~\citep{traud2012social}, following a similar setup to~\citep{lyu2020online}.
	We then used the Markov Chain Monte Carlo (MCMC) algorithm of~\citep{lyu2023sampling} to generate $300$ correlated subgraphs from the networks.
	We then used the resulting matrix as a stream of Markovian data and stopped the algorithms once all the samples are used.
	For comparison, we used the stochastic  majorization-minimization (SMM) algorithm from~\citep{mairal2010online,lyu2020online}, which is the state-of-the-art algorithm for ODL problems. 
	
	In Fig.~\ref{fig:exp}, we see convergence of all the algorithms with respect to the normalized reconstruction error, which is in line with our theoretical results. Moreover, we observe that AdaGrad converges significantly faster than other methods, especially for the sequence of subgraphs from \texttt{Caltech}. The difference in speed of convergence between all methods is marginal for the \texttt{UCLA} and \texttt{Wisconsin}. We suspect that this different behavior is realated to the fact that subgraphs in \texttt{Caltech} induced on random paths of $k=4$ nodes are more likely to contain more edges than those from the other two (much sparser) networks. 

	
	\section{Conclusion} 
	
	In this paper, we have established convergence and complexity results of a wide range of classcial stochastic first-order methods (PSGD, AdaGrad, PSGD-Momentum) under general non-i.i.d. data sampling assumption. Our results show that if the dependence in data samples decays in the length of conditioned steps via MC mixing or Poisson equation, then standard rate of convergence in the i.i.d. case is extended to the more general non-i.i.d. case. Our analysis shows that independence between data samples is not really needed in analyzing stochastic first-order method. We also numerically verified our results on the problem of online dictionary learning from subgraph samples generated by an MCMC algorithm. 

    \section*{Acknowledgements}
    The authors are grateful to Stephen Wright for valuable discussions. The work of A. Alacaoglu was supported by NSF awards 2023239 and 2224213; and DOE ASCR Subcontract 8F-30039 from Argonne National Laboratory. The work of H. Lyu was partially supported by NSF grants DMS-2010035 and DMS-2206296.
    
	\bibliographystyle{apalike}
	\bibliography{mybib}
	
	\newpage
	\appendix
	\onecolumn

	\section{Background on stationarity measures}
	\label{sec:backgrounds}
	
	\subsection{Direct stationarity measures}\label{subsec: direct_stat}
	
	In this subsection, we introduce some notions on stationarity conditions and related quantities. A first-order necessary condition for $\param^{*}\in \Param$ to be a first order stationary point of $f$ over $\Param$ is that there exists a subgradient $v\in \partial f(\param^{*})$ such that $-v$ belongs to the normal cone $N_{\Param}(\param^{*})=\partial \iota_{\Param}(\param^{*})$, which is also equivalent to the variational inequality $\inf_{\param\in \Param} \, \langle v ,\,  \param - \param^{*} \rangle \ge 0$ due to the definition of the normal cone. Hence we introduce the following notion of first-order stationarity for constrained minimization problem:
	\begin{align}
		\textup{$\param^{*}$ is a \textit{stationary point} of $f$ over $\Param$} &\quad \overset{\textup{def}}{\Longleftrightarrow} \quad \mathbf{0}\in v + N_{\Param}(\param^{*}) \quad \textup{for some $v\in \partial f(\param^{*})$} \\
		&\quad \Longleftrightarrow \quad 
		\inf_{\param\in \Param} \, \langle v ,\,  \param - \param^{*} \rangle \ge 0 \quad  \text{for some $v\in \partial f(\param^{*})$}.
	\end{align}
	Note that if $\param^{*}$ is in the interior of $\Param$, then the above is equivalent to $\mathbf{0}\in \partial f(\param^{*})$. Furthermore, if $f$ is differentiable at $\param^{*}$, this is equivalent to $\nabla f(\param^{*})=\mathbf{0}$, so $\param^{*}$ is a \textit{critical point} of $f$. 
	
	In view of the preceding discussion and~\eqref{eq:def_varphi_objective}, we can also say that $\param$ is a stationary point of $f$ over $\Param$ if and only if $\mathbf{0}\in  \partial \varphi(\hat{\param}))$. Accordingly, we may use $\textup{dist}(\mathbf{0}, \partial \varphi(\hat{\param})) = 0$ as an equivalent notion of stationarity. 
	
 We relax the above first-order optimality conditions as follow: For each $\eps>0$, 
	\begin{align}\label{eq:def_eps_stationarity}
		\textup{$\param^{*}$ is an \textit{$\eps$-stationary point} for $f$ over $\Param$} &\quad \overset{\textup{def}}{\Longleftrightarrow} \quad
		\textup{dist}(\mathbf{0}, \partial \varphi(\param^{*}) ) \le \eps.
	\end{align}
	An alternative formulation of $\eps$-stationarity would be using the `stationarity gap'. Namely, we observe the following identity: 
	\begin{align}\label{eq:def_gap_stationarity}
		\textup{Gap}_{\Param}(f,\param^{*}):= \inf_{v\in \partial f(\param^{*})} \left[  -\inf_{\param\in \Param\setminus \{\param^{*}\} }\, \left\langle v,\, \frac{\param-\param^{*}}{\lVert\param-\param^{*} \rVert} \right\rangle \right]=\textup{dist}(\mathbf{0}, \partial \varphi(\param^{*}) ),
	\end{align}
	which is justified in Proposition \ref{prop:approx_stationarity} in Appendix \ref{sec: appendix}. We call the quantity $\textup{Gap}_{\Param}(f,\param^{*})$ above the \textit{stationarity gap} at $\param^{*}$ for $f$ over $\Param$. This measure of approximate stationarity was used in  \citep{lyu2020convergence, lyu2022convergence}, and it is also equivalent to a similar notion in \cite{nesterov2013gradient}. When $\param^{*}$ is in the interior of $\Param$ and if $f$ is differentiable at $\param^{*}$, then \eqref{eq:def_eps_stationarity} is equivalent to $\lVert \nabla f(\param^{*}) \rVert\le \eps$. In Proposition \ref{prop:approx_stationarity}, we provide an equivalent definition of $\eps$-stationarity using the normal cone.

	\section{Preliminary Results}\label{sec: appendix}
	
	The next result illustrates the connection between the two stationarity measures given in~\eqref{eq:def_gap_stationarity} to compare with the existing result in~\citep{lyu2022convergence}.
	Recall that the \textit{normal cone} $N_{\Param}(\param^{*})$ of $\Param$ at $\param^{*}$ is defined as 
	\begin{align}
		N_{\Param}(\param^{*}) :=\left\{ u\in \R^{p}\,|\, \langle u,\, \param-\param^{*} \rangle \le 0 \, \forall \param\in \Param \right\}.
	\end{align}
	Note that the normal cone $N_{\Param}(\param^{*})$ agrees with the subdifferential $\partial \iota_{\Param}(\param^{*})$. When $\Param$ equals the whole space $\R^{p}$, then $N_{\Param}(\param)=\{\mathbf{0}\}$.  
	\begin{proposition}\label{prop:approx_stationarity}
		For each $\param^{*}\in \Param$, $v\in \partial f(\param^{*})$, and $\eps>0$, following conditions are equivalent: 
		\begin{description}[itemsep=0.1cm]
			\item[(i)] $\textup{dist}(\mathbf{0}, v+N_{\Param}(\param^{*}))\le \eps$;
			\item[(ii)] $\displaystyle  -\inf_{\param\in \Param \setminus\{ \param^{*} \} }\, \left\langle v,\, \frac{\param-\param^{*}}{\lVert\param-\param^{*} \rVert} \right\rangle  \le \eps$.
		\end{description}
		In particular, it holds that 
		\begin{align}\label{eq:def_gap_stationarity_pf}
			\textup{dist}(\mathbf{0}, \partial f(\param^{*}) + N_{\Param}(\param^{*}) )= \inf_{v\in \partial f(\param^{*})} \left[  -\inf_{\param\in \Param\setminus\{ \param^{*} \}}\, \left\langle v,\, \frac{\param-\param^{*}}{\lVert\param-\param^{*} \rVert} \right\rangle \right].
		\end{align}
	\end{proposition}    
	
	\begin{proof}
		The last statement follows from the equivalence of \textbf{(i)} and \textbf{(ii)}. In order to show the equivalence, first suppose \textbf{(i)} holds. Then there exists $u\in N_{\Param}(\param^{*})$ and $w\in B_{\eps}$ where $B_{\eps}$ is the $\eps$-ball in $\ell_2$ norm, such that $v+u+w=\mathbf{0}$. So $-v-w\in N_{\Param}(\param^{*})$, which is equivalent to 
		\begin{align}
			\inf_{\param\in \Param\setminus\{ \param^{*} \}}\, \left\langle v+w,\, \frac{\param-\param^{*}}{\lVert\param-\param^{*} \rVert} \right\rangle  \ge 0.
		\end{align}
		By Cauchy-Schwarz inequality, this implies 
		\begin{align}
			-\inf_{\param\in \Param\setminus\{ \param^{*} \}}\, \left\langle v,\, \frac{\param-\param^{*}}{\lVert\param-\param^{*} \rVert} \right\rangle \le \lVert w \rVert \le \eps.
		\end{align}
		
		Conversely, suppose \textbf{(ii)}  holds. We let $\mathcal{D}_{\le 1}(\param^{*})$ denote the set of all feasible directions at $\param^{*}$ of norm bounded by 1, which consists of vectors of form $a(\param-\param^{*})$ for $\param\in \Param$ and $a\in (0,\lVert \param-\param^{*} \rVert^{-1}]$. 
		Being the intersection of two convex sets, $\mathcal{D}_{\le 1}(\param^{*})$ is convex. Then applying the minimax theorem~\cite{sion1958general} for the bilinear map $(x,u)\mapsto \langle v+\eps u, x\rangle$ defined on the product of convex sets $\mathcal{D}_{\le 1}(\param^{*})\times B_{1}$, observe that 
		\begin{align}
			\sup_{u\in B_{1}} \inf_{x\in \mathcal{D}_{\le 1}(\param^{*})} \, \langle v + \eps u ,\, x \rangle &= \inf_{x\in \mathcal{D}_{\le 1}(\param^{*})} \sup_{u\in B_{1}}  \, \langle v + \eps u ,\, x \rangle \label{eq: re3} \\
			&=\inf_{x\in \mathcal{D}_{\le 1}(\param^{*})}\left[   \langle v ,\, x \rangle +\sup_{u\in B_{1}}  \langle \eps u ,\, x \rangle \right]\\
			&=\inf_{x\in \mathcal{D}_{\le 1}(\param^{*})} \left[   \langle v ,\, x \rangle + \eps\lVert x\rVert \right]\\
			&=\inf_{x\in \mathcal{D}_{\le 1}(\param^{*})} \lVert x\rVert \left[   \left\langle v ,\, \frac{x}{\lVert x \rVert} \right\rangle + \eps \right] \ge 0.
		\end{align}
		To see the last inequality, fix $x\in \mathcal{D}_{\le 1}(\param^{*})$. By definition, there exists some $\param_{x}\in \Param$ such that $x/\lVert x \rVert = \frac{\param_{x}-\param^{*}}{\lVert \param_{x} - \param^{*} \rVert}$. Then by using \textbf{(ii)},
		\begin{align}
			\left\langle v ,\, \frac{x}{\lVert x \rVert} \right\rangle + \eps = \left\langle v ,\, \frac{\param_{x}-\param^{*}}{\lVert \param_{x} - \param^{*} \rVert} \right\rangle + \eps \ge \inf_{\param\in \Param\setminus\{ \param^{*} \}}\left\langle v,\, \frac{\param-\param^{*}}{\lVert \param - \param^{*} \rVert} \right\rangle + \eps \ge -\eps + \eps \ge 0.
		\end{align}
		Attainment of the supremum at a $u^*$ in~\eqref{eq: re3} is guaranteed by strong duality, see~\cite{bauschke2011convex}.
		
		The above implies
		\begin{align}
			\inf_{x\in \mathcal{D}_{\le 1}(\param^{*})} \, \langle v + \eps u^{*} ,\, x \rangle  \ge 0.
		\end{align}
		Thus we conclude that $-v-\eps u^{*}\in N_{\Param}(\param^{*})$. Then \textbf{(i)}  holds since $\lVert u^{*} \rVert\le 1$. 
	\end{proof}

	\begin{proposition}\label{prop:stationarty_gap_moreau}
		Suppose $f$ is $\rho$-weakly convex and $\lambda<\rho^{-1}$. Then for each $\param \in \Param$,
		\begin{align}
			\sup_{v\in \partial f(\hat{\param})} \left[ - \inf_{\param'\in \Param \setminus \{ \hat{\param} \} } \left\langle v(\hat{\param}),\,  \frac{\param'-\hat{\param}}{\lVert \param'-\hat{\param} \rVert} \right\rangle  \right] \le \lambda^{-1} \lVert \hat{\param}-\param \rVert \le \lambda^{-2}  \lVert \nabla \varphi_{\lambda}(\param) \rVert 
		\end{align}
	\end{proposition}

	\begin{proof}
		Recall that $\hat{\param}$ is the solution of a constrained optimization problem since $\varphi = f + \iota_{\Param}$~\eqref{eq:proximal_point}.
		Therefore, it satisfies the following first-order optimality condition: For some $v(\hat{\param})\in \partial f(\hat{\param})$, 
		\begin{align}\label{eq:prox_optimality}
			\langle v(\hat{\param}) + \lambda^{-1} (\hat{\param} - \param) ,\, \param'-\hat{\param} \rangle \ge 0, \qquad \forall \param'\in \Param .
		\end{align}
		By rearranging and using Cauchy-Schwarz, this yields for all $\param'\in \Param$, 
		\begin{align}\label{eq:prox_optimality1}
			\langle v(\hat{\param}) ,\, \hat{\param} - \param' \rangle \leq \lambda^{-1} \langle  \hat{\param} - \param ,\, \param'-\hat{\param} \rangle \le \lambda^{-1} \lVert \hat{\param}-\param \rVert\cdot  \lVert \param'-\hat{\param} \rVert,
		\end{align}
		Now assume $\param'\ne \hat{\param}$. Dividing both sides by $\lVert \param'-\hat{\param} \rVert$,  we get 
		\begin{align}
			- \left\langle v(\hat{\param}),\,  \frac{\param'-\hat{\param}}{\lVert \param'-\hat{\param} \rVert} \right\rangle   \le \lambda^{-1} \lVert \hat{\param}-\param \rVert =  \lambda^{-2}  \lVert \nabla \varphi_{\lambda}(\param) \rVert.
		\end{align}
		Since this holds for all $v(\hat{\param})\in \partial f(\hat{\param})$ and $\param'\in \Param\setminus \{\hat{\param} \} $, the assertion follows. 
	\end{proof}

	The next two results will be used in Lem.~\ref{lem:TV_linear_change_bound} to control the bias due to dependent data.
	\begin{lemma}[\cite{rockafellar2009variational}]\label{lem: prox_wc_lips}
		For any $\hat\rho\geq \rho$ and $\rho$-weakly convex function $\varphi$, it follows that $\param\mapsto \prox_{\varphi/\hat\rho}(\param)$ is $\frac{\hat\rho}{\hat\rho-\rho}$-Lipschitz.
	\end{lemma}

	\begin{lemma}[\cite{alacaoglu2020convergence}]\label{lem: bdd_xt_xthat}
		Let $\hat\rho\geq\rho$. Then for any $v\in \partial f(\param)$,
		\begin{equation*}
			\| \hat{\param} - \param \| \leq \frac{2\lVert v \rVert}{\hat\rho-\rho}.
		\end{equation*}
	\end{lemma}

	The following lemma is used in converting various finite total variation results into rate of convergence or asymptotic convergence results. 
	
	\begin{lemma}[Lem. A.5 in \cite{mairal2013stochastic}]\label{lem:positive_convergence_lemma}
		Let $(a_{n})_{n\ge 0}$ and $(b_{n})_{n \ge 0}$ be sequences of nonnegative real numbers such that $\sum_{n=0}^{\infty} a_{n}b_{n} <\infty$. Then the following hold.
		
		\begin{description}
			\item[(i)] $\displaystyle \min_{1\le k\le n} b_{k} \le \frac{\sum_{k=0}^{\infty} a_{k}b_{k}}{\sum_{k=1}^{n} a_{k}}  = O\left( \left( \sum_{k=1}^{n} a_{k} \right)^{-1} \right)$.
			
			\vspace{0.2cm}
			\item[(ii)] Further assume $\sum_{n=0}^{\infty} a_{n} = \infty$ and  $|b_{n+1}-b_{n}|=O(a_{n})$. Then $\lim_{n\rightarrow \infty}b_{n} = 0$.
		\end{description}
	\end{lemma}
	
	\begin{proof}
		\textbf{(i)} follows from noting that
		\begin{align}
			\left( \sum_{k=1}^{n}a_{k} \right) \min_{1\le k \le n} b_{k}\le \sum_{k=1}^{n} a_{k}b_{k} \le  \sum_{k=1}^{\infty} a_{k}b_{k} <\infty.
		\end{align}
		The proof of \textbf{(ii)} is omitted and can be found in \cite{mairal2013stochastic}.
	\end{proof}
	
	The next lemma is commonly used for adaptive gradient algorithms. For example, Lem. A.1 in \cite{levy2017online} or~Lem. 12 in \cite{Duchi:EECS-2010-24}.
	\begin{lemma}[Lem. 12 in \cite{Duchi:EECS-2010-24}, Lem. A.1 in \cite{levy2017online}]\label{lem: num_seq}
		For nonnegative real numbers $a_i$ for $i\geq 1$, we have for any $v_0 > 0$
		\begin{equation*}
			\sum_{i=1}^n \frac{a_i}{v_0+\sum_{j=1}^i a_j} \leq \log\left( 1+\frac{\sum_{i=1}^n a_i}{v_0}\right) \text{~~~~~and~~~~~}
			\sum_{i=1}^n \frac{a_i}{\sqrt{\sum_{j=1}^ia_j}} \leq 2\sqrt{\sum_{i=1}^n a_i}.
		\end{equation*}
	\end{lemma}
	
	The following uniform concentration lemma for vector-valued parameterized observables is due to \cite{lyu2022convergence}.

	\begin{lemma}[Lem 7.1 in \cite{lyu2022convergence}]\label{lem:f_n_concentration_L1_gen}
		Fix compact subsets $\mathcal{X}\subseteq \R^{q}$, $\Param\subseteq \R^{p}$ and a bounded Borel measurable function $\psi:\mathcal{X}\times \Param\rightarrow \R^{r}$. Let $(\x_{n})_{n\ge 1}$ denote a sequence of points in $\mathcal{X}$ such that $\x_{n}=\varphi(X_{n})$ for $n\ge 1$, where $(X_{n})_{n\ge 1}$ is a Markov chain on a state space $\Omega$ and $\varphi:\Omega \rightarrow \mathcal{X}$ is a measurable function. Assume the following: 
		\begin{description}
			\item[(a1)] The Markov chain $(X_{n})_{n\ge 1}$ mixes exponentially fast to its unique stationary distribution and the stochastic process $(\x_{n})_{n\ge 1}$ on $\mathcal{X}$ has a unique stationary distribution $\pi$. 
		\end{description}
		
		\noindent Suppose $w_{n}\in (0,1]$, $n\ge 1$ are non-increasing and satisfy $w_{n}^{-1} - w_{n-1}^{-1}\le 1$ for all $n\ge 1$. Define functions $\bar{\psi}(\cdot):= \E_{\x\sim \pi}\left[ \psi(\x,\cdot) \right] $ and $\bar{\psi}_{n}:\Param\rightarrow \R^{r}$ recursively as $\bar{\psi}_{0}\equiv \mathbf{0}$ and 
		\begin{align}
			\bar{\psi}_{n}(\cdot) = (1-w_{n})\bar{\psi}_{n-1}(\cdot) + w_{n} \psi(\x_{n}, \cdot).
		\end{align}
		Then fthere exists a constant $C>0$ such that for all $n\ge 1$,
		\begin{align}\label{eq:lem_f_fn_bd_gen}
			\sup_{\param\in \Param} \left\lVert \bar{\psi}(\param)- \E[\bar{\psi}_{n}(\param)] \right\rVert  \le Cw_{n}, \quad 	\E\left[ \sup_{\param\in \Param} \left\lVert \bar{\psi}(\param) - \bar{\psi}_{n}(\param) \right\rVert \right] \le Cw_{n}\sqrt{n}.
		\end{align}
		Furthermore, if $w_{n}\sqrt{n}=O(1/(\log n)^{1+\eps})$ for some $\eps>0$, then $\sup_{\param\in \Param} \left\lVert \bar{\psi}(\param) - \bar{\psi}_{n}(\param) \right\rVert\rightarrow 0$ as $t\rightarrow \infty$ almost surely. 
	\end{lemma}

	\section{Proof for Sections~\ref{sec: grad_estim}}\label{sec: app_proofs}

    In this section, we prove the key lemma (Lemma \ref{lem:TV_linear_change_bound}) we stated in the main text. For the reader's convenience, we restate the key lemma here:


  	\begin{lemma}[Lemma \ref{lem:TV_linear_change_bound} in the main text]\label{lem:TV_linear_change_bound_appendix}
   Let Assumptions~\ref{assumption:A2}, ~\ref{assumption:A1},~\ref{assumption:A3} hold and $\param_{t}$ be generated according to Algorithm~\ref{algorithm:prox_grad},~\ref{algorithm:adagrad} or ~\ref{algorithm:shb}.
		Fix $t\ge 0$, $k=k_{t} \in [0,t]$ as in \cref{assumption:A1}, $\hat{\rho}>\rho$  and denote $\hat{\param}=\prox_{\varphi/\hat{\rho}}(\param)$. Then
		\begin{align}
			&\left| \E\left[  \langle \hat{\param}_{t} - \param_{t},\, G(\param_{t}, \x_{t+1}) \rangle \,|\, \mathcal{F}_{t-k}\right] -    \langle \hat{\param}_{t} - \param_{t},\, \E_{\x \sim \pi}\left[ G(\param_{t}, \x) \right] \rangle \right| \\
			&\qquad \le  \frac{{4} L^2}{\hat\rho-\rho} \,   \Delta_{[t-k,t]} 
			+ \frac{2L(L_1+\hat\rho )}{\hat\rho-\rho} \,  \E\left[ \sum_{s=t-k}^{t-1} \alpha_{s}\lVert G(\param_{s},\x_{s+1}) \rVert \,\bigg|\, \mathcal{F}_{t-k} \right].
		\end{align}
	\end{lemma}

    For the proofs in this section, we use the following notations. Let $\pi_{t+1|t-k}=\pi_{t+1|t-k}(\cdot\,|\, \mathcal{F}_{t-k})$ denote the distribution of $\x_{t+1}$ conditional on the information $\mathcal{F}_{t-k} =\sigma(\x_{1},\dots,\x_{t-k})$. Also, $\E_{\x\sim \mu}$ will denote the expectation only with respect to the random variable $\x$ distributed as $\mu$, leaving out any other random variable fixed.

    The following proposition is an important ingredient for the proof of Lemma \ref{lem:TV_linear_change_bound}. It allows us to compare a multi-step conditional expectation of the stochastic gradient to its stationary expectation. 

    \begin{proposition}\label{prop:key_lem_pf}
    Let Assumptions~\ref{assumption:A2},~\ref{assumption:A1},~\ref{assumption:A3} hold and $\param_{t}$ be generated according to Algorithm~\ref{algorithm:prox_grad},~\ref{algorithm:adagrad} or ~\ref{algorithm:shb}. Suppose $\lim_{N\rightarrow\infty} \lVert \pi_{t+N|t} - \pi \rVert_{TV} = 0$ for all $t\ge 0$. 
		Fix $t\ge 0$ and $k\in [0,t]$. Then 
        \begin{align}\label{eq:Lem_claim0}
			\lVert \E[ G(\param_{t-k}, \x_{t+1})\,|\, \mathcal{F}_{t-k} ] - \E_{\x\sim \pi}[G(\param_{t-k}, \x)]   \rVert \le \begin{cases} 2L \Delta_{[t-k, t+1]}  & \textup{if \cref{assumption:A3}(i) holds}; \\
           2L \Delta_{[t-k, t]}  & \textup{if \cref{assumption:A3}(ii) holds}. 
    \end{cases}
		\end{align}
    \end{proposition}

    \begin{proof}
    Recall that by Scheff\'{e}'s lemma, if two probability measures $\mu$ and $\nu$ on the same probability space have densities $\alpha$ and $\beta$ with respect to a reference measure $d m$, then $\lVert \mu-\nu\rVert_{TV} = \frac{1}{2}\int |\alpha-\beta| \, dm$ (see, e.g., Lemma 2.1 in  \cite{tsybakov2004introduction}). For each integer $m\ge 0$, let $\pi'_{t+m|t-k}$ denote the density functions of $\pi_{t+m|t-k}$ with respect to the Lebesgue measure, which we denote by $d\xi$. 

    We first prove the statement under Assumption~\ref{assumption:A3}(ii). In this case, $\lVert G(\param, \x) \rVert$ is assumed to be uniformly bounded by $L<\infty$ but we do not impose any additional assumption on the data samples $(\x_{t})_{t\ge 0}$ besides the asymptotic mixing condition $\lim_{N\rightarrow\infty} \lVert \pi_{t+N|t} - \pi \rVert_{TV} = 0$ for all $t\ge 0$. 
    
    Fix an integer $N\ge 1$. Noting that $\param_{t-k}$ is deterministic with respect to $\mathcal{F}_{t-k}$, we have 
			\begin{align}
				\| \E[ G(\param_{t-k}, \x_{t+1})\,|\, \mathcal{F}_{t-k} ] &- \E_{\x\sim \pi_{t+N|t-k}}[G(\param_{t-k}, \x)] \|  \\
				&=\| \E[ G(\param_{t-k}, \x_{t+1})\,|\, \mathcal{F}_{t-k} ] - \E[G(\param_{t-k}, \x_{t+N})\,|\, \mathcal{F}_{t-k}] \| \\
				&= \| \E_{\x\sim \pi_{t|t-k} }[G(\param_{t-k}, \x)]  -  \E_{\x\sim \pi_{t+N-1|t-k}}[G(\param_{t-k}, \x)] \|\\
				& \le \int_{\Omega} \lVert G(\param_{t-k}, \x) \rVert  | \pi_{t+1|t-k}'(\x) - \pi_{t+N|t-k}'(\x) | \, d\xi\\
				& \le 2 L \lVert \pi_{t+1|t-k}  - \pi_{t+N|t-k} \rVert_{TV},
			\end{align}
			where we have used Scheff\'e's lemma for the last equality. 
			By a similar argument, 
			\begin{align}
				\| \E_{\x\sim \pi}[ G(\param_{t-k}, \x)  ] - \E_{\x\sim \pi_{t+N|t-k}}[G(\param_{t-k}, \x)] \|  
				&\le  \int_{\Omega} \lVert G(\param_{t-k}, \x) \rVert  | \pi'(\x) - \pi_{t+N|t-k}'(\x) | \, d\xi\\
				& \le 2 L \lVert \pi  - \pi_{t+N|t-k} \rVert_{TV}.
			\end{align}
			By triangle inequality, it then follows 
			\begin{align}
				\| \E_{\x\sim \pi}[ G(\param_{t-k}, \x)  ] - \E[ G(\param_{t-k}, \x_{t+1})\,|\, \mathcal{F}_{t-k} ] \| \le 2 L\left(\lVert \pi  - \pi_{t+N|t-k} \rVert_{TV} + \lVert \pi_{t+1|t-k}  - \pi_{t+N|t-k} \rVert_{TV} \right).
			\end{align}
			Now by the hypothesis that  $\lim_{N\rightarrow\infty} \lVert \pi_{t+N|t-k} - \pi \rVert_{TV}=0$, the last expression above converges to $2 L \lVert \pi_{t+1|t-k} - \pi \rVert_{TV}$ as $N\rightarrow \infty$. Since the left hand side above does not depend on $N$, this shows the claim \eqref{eq:Lem_claim0}. 

        Next, we prove the statement under Assumption~\ref{assumption:A3}(i). In this case, we only assume the one-step conditional expectation of the norm of the stochastic gradient is bounded: 
        \begin{align}\label{eq:prop_pf_assumption_one_step_bdd}
            \E[ \lVert G(\param, \x_{t+1} \rVert \,|\, \mathcal{F}_{t}  ] \le L \quad \textup{for all $t\ge 0$}, 
        \end{align}
        which is much weaker than the uniform boundedness  of $\lVert G \rVert$ in Assumption \ref{assumption:A3}(ii). In order to handle a technical difficulty in this general setting, we will need to assume that the data samples $(\x_{t})_{t\ge 0}$ is a function of some time-homogeneous Markov chain. That is, there exists a time-homogeneous Markov chain $(X_{t})_{t\ge 0}$ on some state space $\mathfrak{X}$ and a function $w:\mathfrak{X}\rightarrow \Omega$  such that $\x_{t}=w(X_{t})$ for all $t\ge 0$. By the time-homogeneity of the chain $(X_{t})_{t\ge 0}$, there exists a Markov transition kernel $P$ such that 
        \begin{align}
            \P(X_{t+1}=b\,|\, X_{t}=a) \equiv P(a,b) \qquad \textup{for all $t\ge 0$ and $a,b\in \mathfrak{X}$}. 
        \end{align}

    We will proceed similarly as before. The key technical detail to avoid using uniform boundedness of $G$ is to rewrite expectations of $G$ by the expectation of a one-step conditional expectation of $G$. Then a similar argument as before will work only with the assumption that the one-step conditional expectation of $G$ is bounded. We give the details of this sketched approach below. 

    Fix an integer $N\ge 1$. Since the conditional expectation $\E[ G(\param_{t-k}, \x_{t+m})\,|\, \mathcal{F}_{t-k} ]$ is deterministic with respect to $\mathcal{F}_{t-k}$, we can write 
    \begin{align}\label{eq:prop_key_pf_1}
        \E[ G(\param_{t-k}, \x_{t+1})\,|\, \mathcal{F}_{t-k} ] &=  \E[\E[ G(\param_{t-k}, \x_{t+1})\,|\, \mathcal{F}_{t} ] \,|\, \mathcal{F}_{t-k}].
    \end{align}
    Similarly, write 
    \begin{align}\label{eq:prop_key_pf_2}
    \E[ G(\param_{t-k}, \x_{t+N})\,|\, \mathcal{F}_{t-k} ] &=  \E[\E[ G(\param_{t-k}, \x_{t+N})\,|\, \mathcal{F}_{t+N-1} ] \,|\, \mathcal{F}_{t-k}].
    \end{align}
    Now for given $s\ge 0$, $\param\in \Param$, and $\x\in \Omega$, define 
    \begin{align}
        \tilde{G}_{s}(\param, \x):=\E[G(\param, \x_{s+1})\,|\, \x_{s}=\x].
    \end{align}
    The only randomness being integrated out in the expectation in the above definition is the random data sample $\x_{s+1}$ conditional on the data sample $\x_{s}$ a step before being $\x$. The measure used in the integral is the one-step conditional distribution $\pi_{s+1|s}$. By the time-homogeneity assumption, the distribution $\pi_{s+1|s}$ does not depend on $s$. It follows that the function $\tilde{G}_{s}$ above does not depend on $s$. Therefore, we will omit the subscript $s$ in $\tilde{G}_{s}$. Note that by Jensen's inequality and \eqref{eq:prop_pf_assumption_one_step_bdd}, 
   \begin{align}
       \lVert \tilde{G}(\param, \x) \rVert \le \E[\lVert G(\param, \x_{s+1}) \rVert \,|\, \x_{s}=\x] \le L. 
   \end{align}

   Using \eqref{eq:prop_key_pf_1} and \eqref{eq:prop_key_pf_2}, we have 
			\begin{align}
				\| \E[ G(\param_{t-k}, \x_{t+1})\,|\, \mathcal{F}_{t-k} ] &- \E_{\x\sim \pi_{t+N|t-k}}[G(\param_{t-k}, \x)] \|  \\
				&=\| \E[ G(\param_{t-k}, \x_{t+1})\,|\, \mathcal{F}_{t-k} ] - \E[G(\param_{t-k}, \x_{t+N})\,|\, \mathcal{F}_{t-k}] \| \\
                 &= \| \E_{\x\sim \pi_{t|t-k} }[\tilde{G}_{t}(\param_{t-k}, \x)]  -  \E_{\x\sim \pi_{t+N-1|t-k}}[\tilde{G}_{t+N-1}(\param_{t-k}, \x)] \|\\
				&= \| \E_{\x\sim \pi_{t|t-k} }[\tilde{G}(\param_{t-k}, \x)]  -  \E_{\x\sim \pi_{t+N-1|t-k}}[\tilde{G}(\param_{t-k}, \x)] \|\\
				& = \left\| \int_{\Omega} \tilde{G}(\param_{t-k}, \x) (\pi_{t|t-k}'(\x) - \pi_{t+N-1|t-k}'(\x)) \, d\xi \right\|\label{eq: nm32} \\
				& \le \int_{\Omega} \lVert \tilde{G}(\param_{t-k}, \x) \rVert  | \pi_{t|t-k}'(\x) - \pi_{t+N-1|t-k}'(\x) | \, d\xi\\
				& \le 2 L \lVert \pi_{t|t-k}  - \pi_{t+N-1|t-k} \rVert_{TV},
			\end{align}
			where we have used Scheff\'e's lemma and the fact that $\tilde{G}_{s}$ does not depend on $s$. 
			By a similar argument and noting that $\E_{\x \sim \pi}[G(\param_{t-k}, \x)] = \E_{\x_{t-k} \sim \pi}[G(\param_{t-k}, \x_{t-k+1})] =  \E_{\x_{t-k} \sim \pi}[ \E\left[  G(\param_{t-k}, \x_{t-k+1})]\,|\, \x_{t-k} \right]$, 
			\begin{align}
				\| \E_{\x\sim \pi}[ G(\param_{t-k}, \x)  ] &- \E_{\x\sim \pi_{t+N|t-k}}[G(\param_{t-k}, \x)] \|  \\
				&= \| \E_{\x\sim \pi }[\tilde{G}(\param_{t-k}, \x)]  -  \E_{\x\sim \pi_{t+N-1|t-k}}[\tilde{G}(\param_{t-k}, \x)] \|\\
				& \le 2 L \lVert \pi  - \pi_{t+N-1|t-k} \rVert_{TV}.
			\end{align}
			By triangle inequality, it then follows 
			\begin{align}
				\| \E_{\x\sim \pi}[ G(\param_{t-k}, \x)  ] - \E[ G(\param_{t-k}, \x_{t+1})\,|\, \mathcal{F}_{t-k} ] \| \le 2 L\left(\lVert \pi  - \pi_{t+N-1|t-k} \rVert_{TV} + \lVert \pi_{t|t-k}  - \pi_{t+N-1|t-k} \rVert_{TV} \right).
			\end{align}
            Now by the hypothesis   $\lim_{N\rightarrow\infty} \lVert \pi_{t+N-1|t-k} - \pi \rVert_{TV}=0$, so the last expression above converges to $2 L \lVert \pi_{t|t-k} - \pi \rVert_{TV}$ as $N\rightarrow \infty$.  Since the left hand side above does not depend on $N$, this shows \eqref{eq:Lem_claim0}. 
    \end{proof}

    We now prove Lemma \ref{lem:TV_linear_change_bound}. 
 
	\begin{proof}[Proof of Lemma~\ref{lem:TV_linear_change_bound}]

		 Denote $V(\x,\param):= \langle \hat{\param} - \param,\, G(\param, \x)  \rangle$. Note that $  \E_{\x \sim \pi} \left[ V(\x, \param_{t}) \right] = \langle \hat{\param}_{t} - \param_{t},\, \E_{\x \sim \pi}\left[ G(\param_{t}, \x) \right] \rangle  $. Observe that by triangle inequality
		\begin{align}
			&\left|	\E\left[ V(\x_{t+1}, \param_{t}) \,|\, \mathcal{F}_{t-k} \right] -  \E_{\x \sim \pi} \left[ V(\x, \param_{t}) \right] \right| \\
			&\hspace{3cm}  \le  \left|	\E\left[ V(\x_{t+1}, \param_{t}) -  V(\x_{t+1}, \param_{t-k}) \,|\, \mathcal{F}_{t-k} \right] \right| \\
			&\hspace{4cm} + \left|\E_{\x\sim \pi} \left[ V(\x, \param_{t-k}) -  V(\x, \param_{t}) \right] \right| \\
			&\hspace{4.7cm} +  \left|	\E \left[ V(\x_{t+1}, \param_{t-k}) \,|\, \mathcal{F}_{t-k} \right] -  \E_{\x \sim \pi} \left[ V(\x, \param_{t-k}) \right] \right| .\label{eq:lem_mixing_pf1}
		\end{align}
		We will bound the three terms in the right in order.

			In order to bound the first term in the right hand side above, we first write 
			\begin{multline*}
				V(\x_{t+1}, \param_t) - V(\x_{t+1}, \param_{t-k}) = \langle \hat{\param}_{t} - \param_t, G(\param_t, \x_{t+1})\rangle - \langle \hat{\param}_{t-k} - \param_{t-k}, G(\param_{t-k}, \x_{t+1}) \rangle \\
				= \langle \hat{\param}_t - \param_t, G(\param_t, \x_{t+1}) - G(\param_{t-k}, \x_{t+1}) \rangle + \langle \hat{\param}_{t} - \hat{\param}_{t-k}, G(\param_{t-k}, \x_{t+1}) \rangle \\
				+ \langle \param_{t-k} - \param_t, G(\param_{t-k}, \x_{t+1}) \rangle.
			\end{multline*}
			By applying iterated expectation twice, we get 
			\begin{align}
				&\E\left[  \langle \param_{t-k} - \param_t,\, G(\param_{t-k}, \x_{t+1}) \rangle \,|\, \mathcal{F}_{t-k}\right] \notag \\
				&\qquad = \E_{\param_{t}}\left[ \E\left[   \langle \param_{t-k} - \param_t,\,  G(\param_{t-k}, \x_{t+1})   \rangle \,|\, \param_{t},\, \mathcal{F}_{t-k}\right] \,|\, \mathcal{F}_{t-k}\right]\notag\\
				&\qquad = \E_{\param_{t}}\left[    \left\langle \param_{t-k} - \param_t,\,  \E\left[ G(\param_{t-k}, \x_{t+1})   \,|\, \param_{t},\, \mathcal{F}_{t-k}\right] \right\rangle  \,|\, \mathcal{F}_{t-k}\right]\notag\\
				&\qquad = \E\left[    \left\langle \param_{t-k} - \param_t,\,  \E\left[ G(\param_{t-k}, \x_{t+1})   \,|\, \param_{t},\, \mathcal{F}_{t-k}\right] \right\rangle  \,|\, \mathcal{F}_{t-k}\right].\label{eq: sfm4}
			\end{align}
			We can rewrite the conditional expectation $\E[  \langle \hat{\param}_{t-k} - \hat{\param}_t,\, G(\param_{t-k}, \x_{t+1}) \rangle \,|\, \mathcal{F}_{t-k}]$ similarly as above. 

            Next, we will observe that 
            \begin{align}\label{eq:pf_key_lem_theta_bd}
				\lVert  \hat{\param}_t - \param_t  \rVert  \le \frac{2 }{\hat{\rho}-\rho} \lVert \E_{\x\sim \pi}[G(\param_{t}, \x)] \rVert  \le  \frac{2L }{\hat{\rho}-\rho}.
			\end{align}
            The first inequality above is due to \cref{lem: bdd_xt_xthat}. 
            Under \cref{assumption:A3}(ii), where since $\lVert G \rVert\le L$, the second inequality follows by using Jensen's inequality. In case of \cref{assumption:A3}(i), we need a bit more careful argument. For each $N\ge 1$, 
			\begin{align}
				\lVert  \hat{\param}_t - \param_t  \rVert  
				&\le \frac{2 }{\hat{\rho}-\rho} \lVert \E_{\x\sim \pi}[G(\param_{t}, \x)] \rVert \\
                & \le \frac{2 }{\hat{\rho}-\rho} \left( \lVert  \E_{\x\sim \pi_{t+N|t}}[ G(\param_{t}, \x)  ] \rVert +  \lVert  \E_{\x\sim \pi}[G(\param_{t}, \x)  -   \E_{\x\sim \pi_{t+N|t}}[ G(\param_{t}, \x)  ] \rVert  \right) \\
                & \le \frac{2 }{\hat{\rho}-\rho} \left(   \E[\lVert G(\param_{t}, \x_{t+N}) \rVert \,|\, \mathcal{F}_{t} ] +  \lVert  \E_{\x\sim \pi}[G(\param_{t}, \x)  -   \E_{\x\sim \pi_{t+N|t}}[ G(\param_{t}, \x)  ] \rVert  \right).
			\end{align}
            Note that $ \E[\lVert G(\param_{t}, \x_{t+N}) \rVert \,|\, \mathcal{F}_{t} ]\le L$ by \cref{assumption:A3}(ii) and iterated expectation. Furthermore, the second term in the last expression above vanishes as $N\rightarrow\infty$ by  Proposition \ref{prop:key_lem_pf} and \cref{assumption:A1}. Therefore we can conclude \eqref{eq:pf_key_lem_theta_bd} under \cref{assumption:A2}(ii) as well. 

            Now  by using Cauchy-Schwarz inequality, $L_{1}$-Lipschitz continuity of $\param\mapsto G(\x,\param)$ (see \cref{assumption:A2}),  Lemma~\ref{lem: prox_wc_lips} and~\cref{assumption:A3}, we obtain
			\begin{align}
				| \E[ V(\x_{t+1}, \param_{t}) - V(\x_{t+1}, \param_{t-k}) \,|\, \mathcal{F}_{t-k} ] | &\le  \frac{2L L_1+2\hat\rho L}{\hat\rho-\rho} \,  \E\left[ \lVert \param_{t-k} - \param_{t} \rVert \,|\, \mathcal{F}_{t-k} \right] \\
				&\hspace{-1.5cm}\le \frac{2L L_1+2\hat\rho L}{\hat\rho-\rho} \,  \E\left[ \sum_{s=t-k}^{t-1} \alpha_{s} \lVert G(\param_{s},\x_{s+1}) \rVert \,\bigg|\, \mathcal{F}_{t-k}, \right],\label{eq: eti3}
			\end{align}
			where for the last inequality we have used  $\lVert \param_{s}-\param_{s-1} \rVert \le \lVert \alpha_{s-1} G(\param_{s-1}, \x_{s}) \rVert$ for $s\ge 1$ along with triangle inequality.

		A similar argument shows 
		\begin{equation}
			|	\E_{\x\sim \pi}\left[ V(\x, \param_{t}) - V(\x, \param_{t-k})  \right] |
			\le  \frac{2L L_1+2\hat\rho L}{\hat\rho-\rho} \,  \E\left[ \sum_{s=t-k}^{t-1} \alpha_{s}\lVert G(\param_{s-1},\x_{s}) \rVert \,|\, \mathcal{F}_{t-k} \right].\label{eq: eti4}
		\end{equation}
		We continue by estimating the last term on the RHS of~\eqref{eq:lem_mixing_pf1}. 
  
        Proceeding by Cauchy-Schwarz inequality, and using that $\param_{t-k}, \hat{\param}_{t-k}$ are deterministic with respect to $\mathcal{F}_{t-k}$, we get 
		\begin{align}\label{eq:inner_prod_approx}
			&\left|\E\left[ V(\x_{t+1}, \param_{t-k}) \,|\, \mathcal{F}_{t-k} \right]   - \E_{\x \sim \pi}\left[ V(\x, \param_{t-k})  \right]  \right| \\
			&\qquad = \left|\E\left[ \langle  \hat{\param}_{t-k} - \param_{t-k},\, G(\param_{t-k},\x_{t+1})  \rangle \,|\, \mathcal{F}_{t-k} \right]   - \E_{\x \sim \pi}\left[ V(\x, \param_{t-k})  \right]  \right| \\
			&\qquad = \left| \langle  \hat{\param}_{t-k} - \param_{t-k},\, \E\left[ G(\param_{t-k},\x_{t+1})\,|\, \mathcal{F}_{t-k} \right]   \rangle   - \langle  \hat{\param}_{t-k} - \param_{t-k},\, \E_{\x\sim \pi}\left[ G(\param_{t-k},\x) \right]   \rangle   \right| \\
			&\qquad \le 2 L \lVert \hat{\param}_{t-k} - \param_{t-k} \rVert \,  \Delta_{[t-k,t]}\} \\
            &\qquad \leq \frac{4 L^2}{\hat\rho-\rho} \Delta_{[t-k,t]} ,\label{eq: eti5}
		\end{align}
		where the last step follows from \cref{prop:key_lem_pf}. 
		Combining~\eqref{eq: eti3},~\eqref{eq: eti4},~\eqref{eq: eti5} with \eqref{eq:lem_mixing_pf1} then shows the assertion. 
	\end{proof}
	
\section{Proof for Section~\ref{sec: proj_sgd}}\label{app: psgd}
		
	\begin{theorem}[Theorem~\ref{th: th_proj_sgd} in the main text]\label{th: th_proj_sgd_appendix}
Let Assumptions \ref{assumption:A2}-\ref{assumption:A3} hold and $(\param_{t})_{t\ge 1}$ be a sequence generated by Algorithm \ref{algorithm:prox_grad}. 
		Fix $\hat{\rho}>\rho$. Then the following hold:
		\begin{description}
			\item[(i)] (Rate of convergence)  For each $T\ge 1$, 
			{\small
				\begin{align}
					&E\left[     \lVert \nabla \varphi_{1/\hat{\rho}}(\param_{T}^{\mathrm{out}}) \rVert^{2} \right]  \\
					& \quad  \le \frac{\hat{\rho}^{2}L^2}{\hat{\rho}-\rho} \frac{1}{\sum_{k=1}^{T} \alpha_{k}}\bigg[\frac{ \varphi_{1/\hat{\rho}}(\param_{1}) - \inf \varphi_{1/\hat{\rho}}}{\hat{\rho}L^2}   + \frac{1}{2} \sum_{t=1}^{T} \alpha_{t}^{2} + \frac{2(L_1 + \hat\rho)}{\hat\rho-\rho} \sum_{t=1}^{T} k_{t}\alpha_{t}\alpha_{t-k_{t}} + \frac{4}{\hat\rho-\rho} \sum_{t=1}^{T} \alpha_{t} \E[\Delta_{[t-k_{t}, t]} ]\bigg].
				\end{align}
			}
In particular, with $\alpha_t = \frac{c}{\sqrt{t}}$ for some $c> 0$ and under exponential mixing, we have that $ \E\left[     \lVert \nabla \varphi_{1/\hat{\rho}}(\param_{T}^{\mathrm{out}}) \rVert \right] \leq \varepsilon $ with $\tilde{O}\left( \varepsilon^{-4} \right)$ samples.

			\item[(ii)] (Global convergence) Further assume that $\sum_{t=0}^{\infty} k_{t}\alpha_{t} \alpha_{t-k_{t}}<\infty$. Then $\lVert \nabla \varphi_{1/\hat{\rho}}(\hat{\param}_{t}) \rVert \rightarrow 0$ as $t\rightarrow \infty$ almost surely. Furthermore, $\param_{t}$ converges to the set of all stationary points of $f$ over $\Param$.
		\end{description}
	\end{theorem}

	\begin{proof}
		Recall the definition of $\varphi_{1/\hat{\rho}}$ from~\eqref{eq: def_moreu_env}.
		We start as in~\citep{davis2019stochastic} with the difference of conditoning on $\mathcal{F}_{t-k}$ instead of the latest iterate, and use Lemma~\ref{lem:TV_linear_change_bound} to handle the additional bias due to dependent sampling.
		Denote $\hat{\param}_t=\prox_{\varphi/\hat{\rho}}(\param_t)$ for $t\ge 1$ and fix $k\in \{0,\dots,t\}$. Observe that 
		\begin{align}
			\E\Big[ &\varphi_{1/\hat{\rho}} (\param_{t+1}) \,\bigg|\, \mathcal{F}_{t-k} \Big] \le \E\left[ f(\hat{\param}_{t}) + \frac{\hat{\rho}}{2} \lVert \param_{t+1} - \hat{\param}_{t}\rVert^{2} \,\bigg|\, \mathcal{F}_{t-k} \right]  \\
			&= \E\left[ f(\hat{\param}_{t}) \,\bigg|\, \mathcal{F}_{t-k} \right] + \frac{\hat{\rho}}{2} \E\left[  \left\lVert \textup{proj}_{\Param}(\param_{t} - \alpha_{t} G(\param_{t}, \x_{t+1})) -\textup{proj}_{\Param}(\hat{\param}_{t}) \right\rVert^{2} \,\bigg|\, \mathcal{F}_{t-k} \right] \label{eq: wej4}  \\
			&\le \E\left[ f(\hat{\param}_{t}) \,\bigg|\, \mathcal{F}_{t-k} \right] + \frac{\hat{\rho}}{2} \E\left[  \left\lVert (\param_{t}-\hat{\param}_{t}) - \alpha_{t} G(\param_{t}, \x_{t+1}) \right\rVert^{2} \,\bigg|\, \mathcal{F}_{t-k} \right]   \\
			&\le \E\left[ f(\hat{\param}_{t})  + \frac{\hat{\rho}}{2}  \lVert \param_{t} - \hat{\param}_{t}\rVert^{2} \,\bigg|\, \mathcal{F}_{t-k} \right] + \hat{\rho} \alpha_{t} \E\left[ \langle \hat{\param}_{t} - \param_{t},\, G(\param_{t}, \x_{t+1})\rangle \,\bigg| \, \mathcal{F}_{t-k} \right]  + \frac{\alpha_{t}^{2} \hat{\rho} L^{2} }{2}\label{eq: htr3} \\
			&\le \E\left[ \varphi_{1/\hat{\rho}}(\param_{t}) \,\bigg|\, \mathcal{F}_{t-k}\right] + \hat{\rho} \alpha_{t}  \langle \hat{\param}_{t} - \param_{t},\, \E_{\x\sim \pi}\left[ G(\param_{t}, \x) \right] \rangle  + \frac{\alpha_{t}^{2} \hat{\rho} L^{2} }{2}\\
			&\qquad + \hat\rho\alpha_t \left( \frac{4L^2}{\hat\rho-\rho} \, \E\left[   \Delta_{[t-k,t]}  \right] + \frac{2L( L_1+\hat\rho )}{\hat\rho-\rho} \alpha_{t-k} {\sum_{s=t-k}^{t-1} \E\left[ \lVert G(\param_{s},\x_{s+1}) \rVert \,|\, \mathcal{F}_{t-k} \right] }\right).\label{eq: rot4}
		\end{align}
		Namely, the first and the last  inequalities use the definition of Moreau envelope $\varphi_{1/\hat\rho}$ and $\hat{\param}_{t}\in \Param$, the second inequality uses $1$-Lipschitzness of the projection operator, and the last inequality uses Lemma \ref{lem:TV_linear_change_bound} and that $\alpha_{s}$ is non-increasing in $s$. {Note that using iterated expectation, \cref{assumption:A3}, and the fact that $\param_{s}$ is deterministic with respect to $\mathcal{F}_{s}$, for each $t-k\le s \le t-1$, we get 
			\begin{align}
				\E\left[ \lVert G(\param_{s},\x_{s+1}) \rVert \,|\, \mathcal{F}_{t-k} \right] &= \E\left[ \E\left[  \lVert G(\param_{s},\x_{s+1}) \rVert \,|\, \mathcal{F}_{s} \right] \,|\, \mathcal{F}_{t-k}  \right] \le L.
			\end{align}
			Hence the summation in the last term above is bounded above by $kL$.} Then by using \cref{assumption:A2} and the weak convexity of $g$, we have
		\begin{align}\label{eq: fke3}
			\langle \hat{\param}_{t} - \param_{t},\, \E_{\x\sim \pi}\left[ G(\param_{t}, \x) \right] \rangle \leq  f(\hat{\param}_{t}) - f(\param_{t}) + \frac{\rho}{2} \lVert \param_{t} - \hat{\param}_{t} \rVert^{2} .
		\end{align}
		By using this estimate in~\eqref{eq: rot4} and then integrating out $\mathcal{F}_{t-k}$, we get 
		\begin{align}
			\E\left[ \varphi_{1/\hat{\rho}} (\param_{t+1}) \right] -  \E\left[ \varphi_{1/\hat{\rho}}(\param_{t})\right] &\le \hat{\rho} \alpha_{t} \E\left[   f(\hat{\param}_{t}) - f(\param_{t}) + \frac{\rho}{2} \lVert \param_{t} - \hat{\param}_{t} \rVert^{2}   \right] + \frac{\alpha_{t}^{2} \hat{\rho} L^{2} }{2} \\
			&\qquad + \hat\rho\alpha_t \left( \frac{4 L^2}{\hat\rho-\rho} \, \E[   \Delta_{[t-k,t]} ]  + k\frac{2L^2 (L_1+\hat\rho)}{\hat\rho-\rho} \alpha_{t-k}\right).
		\end{align}
		Now we chose $k=k_{t}\rightarrow \infty$ as $t\rightarrow \infty$. Summing over $t=1,\dots,T$ results in
		\begin{align}
			\hat{\rho} \sum_{t=1}^{T}  \alpha_{t} \E\Big[   f(\param_{t}) &- f(\hat{\param}_{t}) -  \frac{\rho}{2}  \lVert \param_{t} - \hat{\param}_{t} \rVert^{2}  \Big] \le \left( \varphi_{1/\hat{\rho}}(\param_{1}) - \inf \varphi_{1/\hat{\rho}}  \right)+ \frac{\hat{\rho}L^{2}}{2} \sum_{t=1}^{T} \alpha_{t}^{2} \\
			&+ \frac{4\hat\rho L^2}{\hat\rho-\rho} \sum_{t=1}^{T} \alpha_{t} \E[\Delta_{[t-k_{t}, t]}]   + \frac{2L^2\hat\rho( L_1 + \hat\rho)}{\hat\rho-\rho} \sum_{t=1}^{T} k_{t}\alpha_{t}\alpha_{t-k_{t}}. \label{eq:thm1_bound1}
		\end{align}
		Next, we use the fact that the function $\param\mapsto f(\param) + \frac{\hat{\rho}}{2}\lVert \param-\param_{t} \rVert^{2}$ is strongly convex with parameter $(\hat{\rho}-\rho)/2$ that is minimized at $\hat{\param}_{t}$ to get 
		\begin{align}
			f(\param_{t}) - f(\hat{\param}_{t}) - \frac{\rho}{2}\lVert \param_{t} - \hat{\param}_{t} \rVert^{2} & = \left(f(\param_{t}) + \frac{\hat\rho}{2}\lVert \hat{\param}_{t} - \param_{t} \rVert^{2} \right) - \left(f(\hat{\param}_{t}) + \frac{\hat{\rho}}{2}\lVert \hat{\param}_{t} - \param_{t} \rVert^{2} \right)  + \frac{\hat{\rho}-\rho}{2} \lVert \param_{t} - \hat{\param}_{t} \rVert^{2} \\
			&  \ge (\hat{\rho} - \rho) \lVert \param_{t} - \hat{\param}_{t} \rVert^{2} \\
			& = \frac{\hat{\rho} - \rho}{\hat{\rho}^{2}} \lVert \nabla \varphi_{1/\hat{\rho}}(\param_{t}) \rVert^{2}.\label{eq: siw3}
		\end{align}
		where the second to the last equality uses \eqref{eq:moreau_grad}. 
		Combining with \eqref{eq:thm1_bound1}, this implies  \begin{align}
			\frac{\hat{\rho} - \rho}{\hat{\rho}}  \sum_{t=1}^{T}  \alpha_{t} \E\left[   \lVert \nabla \varphi_{1/\hat{\rho}}(\param_{t}) \rVert^{2}  \right] &\le \left( \varphi_{1/\hat{\rho}}(\param_{1}) - \inf \varphi_{1/\hat{\rho}}  \right)+ \frac{\hat{\rho}L^{2}}{2} \sum_{t=1}^{T} \alpha_{t}^{2} \\
			& + \frac{4 \hat\rho L^2}{\hat\rho-\rho} \sum_{t=1}^{T} \alpha_{t} \E[   \Delta_{[t-k_{t},t]} ]
			+ \frac{2L^2\hat\rho( L_1 + \hat\rho)}{\hat\rho-\rho} \sum_{t=1}^{T} k_{t}\alpha_{t}\alpha_{t-k_{t}}.\label{eq:thm1_bound2}
		\end{align}
		This shows the assertion when $\param_{T}^{\textup{out}}=\param_{\tau}$. If $\param_{T}^{\textup{out}}\in\arg\min_{\param\in\{\param_1, \dots, \param_T \}} \| \nabla \varphi_{1/\hat{\rho}}(\param) \|^2$, the assertion follows from \eqref{eq:thm1_bound2} and Lemma \ref{lem:positive_convergence_lemma} in  Appendix \ref{sec: appendix}.
		
		For the second part of \textbf(i), we plug in the value of $\alpha_t$ and $k_t = O(\log t)$, $  \Delta_{[t-k_{t},t]}  = O(\lambda^{k_t})$ for $\lambda \in (0, 1)$ under the exponential mixing assumption.
		
		Next, we show \textbf{(ii)}. We will first show that $\lVert \nabla \varphi_{1/\hat{\rho}}(\param_{t}) \rVert\rightarrow 0$ almost surely as $t\rightarrow \infty$. Under the hypothesis, by \eqref{eq:thm1_bound2}, we have 
		\begin{align}
			\sum_{t=1}^{\infty} \alpha_{t} \E\left[\lVert \nabla \varphi_{1/\hat{\rho}}(\param_{t}) \rVert^{2} \right] <\infty.
		\end{align}
		By Fubini's theorem, this implies 
		\begin{align}
			\sum_{t=1}^{\infty} \alpha_{t} \lVert \nabla \varphi_{1/\hat{\rho}}(\param_{t}) \rVert^{2}  <\infty \quad \textup{almost surely.}
		\end{align}
		We will then use Lemma \ref{lem:positive_convergence_lemma} \textbf{(ii)} to deduce that $\lVert \nabla \varphi_{1/\hat{\rho}}(\param_{t}) \rVert\rightarrow 0$ almost surely as $t\rightarrow \infty$. To this end, it suffices to verify  
		\begin{align}\label{eq:pf_moreau_lipschitz}
			\left|  \lVert  \nabla \varphi_{1/\hat{\rho}}(\param_{t+1}) \lVert^{2} - \lVert \nabla \varphi_{1/\hat{\rho}}(\param_{t}) \rVert^{2}   \right| = O(\alpha_{t}).
		\end{align}
		Indeed, by using \eqref{eq:moreau_grad} and Lemma \ref{lem: prox_wc_lips} in Appendix \ref{sec: appendix}, 
		\begin{align}
			\frac{1}{\hat{\rho}} \lVert  \nabla \varphi_{1/\hat{\rho}}(\param_{t+1})-\nabla \varphi_{1/\hat{\rho}}(\param_{t}) \rVert &\le   \lVert \param_{t+1}-\param_{t} \rVert + \lVert \prox_{\varphi/\hat{\rho}}(\param_{t+1}) - \prox_{\varphi/\hat{\rho}}(\param_{t}) \rVert \\
			&\le \frac{2\hat{\rho}-\rho}{\hat{\rho}-\rho}   \lVert \param_{t+1}-\param_{t} \rVert \\
			&= \frac{2\hat{\rho}-\rho}{\hat{\rho}-\rho}   \lVert \textup{proj}_{\Param}(\param_{t}-\alpha_{t} G(\param_{t},\x_{t+1}))-\textup{proj}_{\Param}(\param_{t}) \rVert \\
			&\le \alpha_{t} \frac{2\hat{\rho}-\rho}{\hat{\rho}-\rho} L,
		\end{align}
		where the last inequality uses \cref{assumption:A3}.
		This estimate and Lemma \ref{lem: bdd_xt_xthat} imply
		\begin{align}
			|  \lVert  \nabla \varphi_{1/\hat{\rho}}(\param_{t+1}) \lVert^{2} &- \lVert \nabla \varphi_{1/\hat{\rho}}(\param_{t}) \rVert^{2}|  \\
			&\le   \lVert  \nabla \varphi_{1/\hat{\rho}}(\param_{t+1}) - \nabla \varphi_{1/\hat{\rho}}(\param_{t}) \rVert  \left(  \lVert  \nabla \varphi_{1/\hat{\rho}}(\param_{t+1}) \lVert + \lVert \nabla \varphi_{1/\hat{\rho}}(\param_{t}) \rVert\right) \\
			&\le \alpha_{t} \frac{2\hat{\rho}-\rho}{\hat{\rho}-\rho}  \frac{4L^{2}}{\hat\rho-\rho}.
		\end{align}
		Hence \eqref{eq:pf_moreau_lipschitz} follows, as desired.
		
		Finally, assume $f$ is continuously differentiable. Choose a subsequence $t_{k}$ such that $\hat{\param}_{t}$ converges to some limit point $\hat{\param}_{\infty}$. We will argue that $\param_{t}\rightarrow \hat{\param}_{\infty}$ almost surely as $t\rightarrow \infty$ and $\hat{\param}_{\infty}$ is a stationary point of $f$ over $\Param$. By  \eqref{eq:Moreau_near_stationary} and the first part of \textbf{(ii)}, it holds that $\lVert \hat{\param}_{t} - \param_{t} \rVert + \textup{dist}(\mathbf{0}, \partial \varphi(\hat{\param}_{t}))\rightarrow 0$ almost surely as $t\rightarrow \infty$. By triangle inequality $ \| \hat{\param}_{\infty} - \param_t \| \leq \| \hat{\param}_{\infty} - \hat{\param}_t \| + \|\hat{\param}_t - \param_t \|$, this implies $\hat{\param}_t\rightarrow \hat{\param}_{\infty}$. 
		
		Next, fix arbitrary $\param\in \Param \setminus \{ \hat{\param}_{\infty} \}$. Since $\hat{\param}_{t}\rightarrow \hat{\param}_{\infty}\ne \param$, it holds that $\param\ne \hat{\param}_{t}$ for all sufficiently large $t$. Note that 
		\begin{multline}
			\left| \left\langle \nabla f(\hat{\param}_{\infty}),\, \frac{\param-\hat{\param}_{\infty}}{\lVert \param - \hat{\param}_{\infty} \rVert} \right\rangle -  \left\langle \nabla f(\hat{\param}_{t}),\, \frac{\param-\hat{\param}_{t}}{\lVert \param - \hat{\param}_{t} \rVert} \right\rangle \right| \le \lVert \nabla f(\hat{\param}_{\infty}) - \nabla f (\hat{\param}_{t}) \rVert  \\
			+ \left| \left\langle \nabla f(\hat{\param}_{\infty}),\, \frac{\param-\hat{\param}_{\infty}}{\lVert \param - \hat{\param}_{\infty} \rVert} - \frac{\param-\hat{\param}_{t}}{\lVert \param - \hat{\param}_{t} \rVert} \right\rangle \right|.
		\end{multline}
		The last term tends to zero since $\hat{\param}_{t}\rightarrow \hat{\param}_{\infty}$ and the function $\param'\mapsto \frac{\param-\param'}{\lVert \param-\param'\rVert}$ is continuous whenever $\param'\ne \param$. Also, since $\nabla f$ is continuous and $\hat{\param}_{t}\rightarrow \hat{\param}_{\infty}$, the first term also tends to zero as $t\rightarrow \infty$. Then by using the relation  \eqref{eq:def_gap_stationarity}, we get 
		\begin{align}
			\left\langle \nabla f(\hat{\param}_{\infty}),\, \frac{\param-\hat{\param}_{\infty}}{\lVert \param - \hat{\param}_{\infty} \rVert} \right\rangle  \ge  \left\langle \nabla f(\hat{\param}_{t}),\, \frac{\param-\hat{\param}_{t}}{\lVert \param - \hat{\param}_{t} \rVert} \right\rangle  - o(1)  \ge -\textup{dist}(\mathbf{0},  \partial \varphi(\hat{\param}_{t}))  - o(1)
		\end{align}
		for all sufficiently large $t\ge 1$. by the first part of \textbf{(ii)} and  \eqref{eq:Moreau_near_stationary}, we have $\textup{dist}(\mathbf{0}, \partial \varphi(\hat{\param}_{t}))\rightarrow 0$ as $t\rightarrow \infty$. But since the left hand side does not depend on $t$, it implies that the left hand side above is nonnegative. As $\param\in\Param \setminus \{\hat{\param}_{\infty} \}$ is arbitrary, we conclude that $\hat{\param}_{\infty}$ is a stationary point of $f$ over $\Param$.
	\end{proof}
	
	
 	\section{Proof for Section~\ref{sec: adagrad}}\label{app: adagrad_appendix}	
  
	\begin{theorem}[Theorem~\ref{th: th_adagrad} in the main text]\label{th: th_adagrad_appendix}
		Let \cref{assumption:A2}-\ref{assumption:A3} and \cref{assumption:A4} hold and $(\param_{t})_{t\ge 1}$ be a sequence generated by Algorithm \ref{algorithm:adagrad}. 
		Fix $\hat{\rho}>\rho$ and a nondecreasing, diverging sequence $(k_{t})_{t\ge 1}$. 
		Then,
		for each $T\ge 1$,
		{\small
			\begin{align}
				\E\left[     \lVert \nabla \varphi_{1/\hat{\rho}}(\param_{T}^{\mathrm{out}}) \rVert^{2} \right]   &\le \frac{\hat\rho^2 L}{T(\hat\rho-\rho)} \Bigg( \frac{C_\varphi\sqrt{v_0+TL^2}}{\alpha\hat\rho L} + \sqrt{T} \\
				&+ \frac{2(L_1 + \hat\rho)}{\hat\rho-\rho} \big( \sqrt{T}k_T + \frac{\sqrt{T}k_T\alpha^2}{2}\log(1+v_0^{-1}TL^2)\big)  + \frac{2L}{\hat\rho-\rho}\sum_{t=1}^T \mathbb{E}[\Delta_{[t-k_t, t+1]}]\Bigg) \\
				&=O\left(\frac{k_T \log(TL^2)}{\sqrt{T}} +  \frac{1}{T}\sum_{t=1}^{T} \mathbb{E}[\Delta_{[t-k_{t}, t+1]}]\right).
			\end{align}
		}
	\end{theorem}
 
	\begin{proof}[Proof of Theorem~\ref{th: th_adagrad}]
		We proceed as the proof of Thm.~\ref{th: th_proj_sgd}, but with the difference that $\alpha_t$ is random and depends on the history of observed stochastic gradients, with $G(\param_t, \x_{t+1})$ being the last stochastic gradient that $\alpha_t$ depends on.
		
		We estimate as in the first chain of inequalities in the proof of Thm.~\ref{th: th_proj_sgd} with $\alpha_t$ dividing both sides and by omitting the expectation because of the randomness of $\alpha_t$.
		In particular, we have
		\begin{align}
			\frac{1}{\alpha_t} \varphi_{1/\hat{\rho}} (\param_{t+1}) &\le \frac{1}{\alpha_t}\left[f(\hat{\param}_{t}) + \frac{\hat{\rho}}{2} \lVert \param_{t+1} - \hat{\param}_{t}\rVert^{2}\right]  \\
			&= \frac{1}{\alpha_t}\left[ f(\hat{\param}_{t}) + \frac{\hat{\rho}}{2}  \left\lVert \textup{proj}_{\Param}(\param_{t} - \alpha_{t} G(\param_{t}, \x_{t+1})) -\textup{proj}_{\Param}(\hat{\param}_{t}) \right\rVert^{2} \right]   \\
			&\le \frac{1}{\alpha_t}\left[ f(\hat{\param}_{t}) + \frac{\hat{\rho}}{2}  \left\lVert (\param_{t}-\hat{\param}_{t}) - \alpha_{t} G(\param_{t}, \x_{t+1}) \right\rVert^{2}  \right]   \\
			&\le \frac{1}{\alpha_t}\left[ f(\hat{\param}_{t})  + \frac{\hat{\rho}}{2}  \lVert \param_{t} - \hat{\param}_{t}\rVert^{2} \right] + \hat{\rho} \langle \hat{\param}_{t} - \param_{t},\, G(\param_{t}, \x_{t+1})\rangle   + \frac{\alpha_{t} \hat{\rho} \|G(\param_t, \x_{t+1})\|^2 }{2} \\
			&=\frac{1}{\alpha_t} \varphi_{1/\hat{\rho}}(\param_{t}) + \hat{\rho} \langle \hat{\param}_{t} - \param_{t},\, G(\param_{t}, \x_{t+1})\rangle   + \frac{\alpha_{t} \hat{\rho} \|G(\param_t, \x_{t+1})\|^2 }{2}.\label{eq: roit4}
		\end{align}
		Proceeding as in the proof of Theorem \ref{th: th_proj_sgd}, namely, by taking expectation conditional on $\mathcal{F}_{t-k}$, using Lemma \ref{lem:TV_linear_change_bound}, using~\eqref{eq: fke3}, and then integrating $\mathcal{F}_{t-k}$ out, we  obtain
		\begin{multline}
			\E\left[ \frac{1}{\alpha_t}\varphi_{1/\hat{\rho}} (\param_{t+1}) \right] -  \E\left[\frac{1}{\alpha_t} \varphi_{1/\hat{\rho}}(\param_{t})\right] \le \hat{\rho} \E\left[   f(\hat{\param}_{t}) - f(\param_{t}) + \frac{\rho}{2} \lVert \param_{t} - \hat{\param}_{t} \rVert^{2}   \right] \\
			+ \mathbb{E} \left[\frac{\alpha_{t} \hat{\rho} \| G(\param_t, \x_{t+1}) \|^2 }{2}\right] \\
			+ \hat\rho \mathbb{E}\left[ \frac{{4}L^2}{\hat\rho-\rho} \, \Delta_{[t-k,t]}  + \frac{2L(L_1+\hat\rho) }{\hat\rho-\rho} \sum_{s=t-k}^{t-1} \alpha_s \|G(\param_{s}, x_{s+1})\|. \right].
		\end{multline}
		The only difference from before is that while bounding $\|\param_s - \param_{s-1} \|$ in Lem.~\ref{lem:TV_linear_change_bound} we did not use the worst case bound for $\|G(\param_{s}, x_{s+1})\|$ as in \cref{assumption:A2}.
		
		We use~\eqref{eq: siw3} on this inequality with $k=k_{t}$ where $k_{t}$ is nondecreasing, sum for $t\in \{1, 2, \dots, T\}$ and rearrange to get
		\begin{multline}
			\frac{\hat\rho-\rho}{\hat\rho} \sum_{t=1}^{T}\mathbb{E}\left[ \| \nabla \varphi_{1/\hat\rho}(\param_t) \|^2 \right]\leq \sum_{t=1}^{T} \mathbb{E}\left[ \frac{\varphi_{1/\hat\rho}(\param_{t})-\varphi_{1/\hat\rho}(\param_{t+1})}{\alpha_t} \right] + \sum_{t=1}^{T}\mathbb{E} \left[\frac{\alpha_{t} \hat{\rho} \| G(\param_t, \x_{t+1}) \|^2 }{2}\right] \\
			+ \sum_{t=1}^{T}\hat\rho \mathbb{E}\left[ \frac{ {4} L^2}{\hat\rho-\rho} \, \Delta_{[t-k_{t},t]}  + \frac{2L(L_1+\hat\rho) }{\hat\rho-\rho} \sum_{s=t-k_{t}}^{t-1} \alpha_s \|G(\param_{s}, x_{s+1})\| \right].\label{eq: wie3}
		\end{multline}
		We continue to upper bound the terms on the RHS of this inequality.
		We use Lem.~\ref{lem: num_seq} to bound
		\begin{equation}\label{eq: ig3}
			\sum_{t=1}^{T} \alpha_t \| G(\param_t, \x_{t+1}) \|^2 = \sum_{t=1}^{T} \frac{\alpha}{\sqrt{v_0+\sum_{j=1}^t \| G(\param_j, \x_{j+1}) \|^2}} \| G(\param_t, \x_{t+1}) \|^2
			\leq 2\sqrt{\sum_{t=1}^{T} \| G(\param_t, \x_{t+1})\|^2 }, 
		\end{equation}
  {where we also used that $v_0 > 0$.}
		By taking expectation, and using Jensen's inequality, we get 
		\begin{equation}\label{eq:ig31}
			\E\left[ \sum_{t=1}^{T} \alpha_t \| G(\param_t, \x_{t+1}) \|^2 \right]  \leq \E\left[  2\sqrt{\sum_{t=1}^{T} \| G(\param_t, \x_{t+1})\|^2 } \right] \le 2\sqrt{ \sum_{t=1}^{T} \E\left[\| G(\param_t, \x_{t+1})\|^2 \right]  }
			\le 2 \sqrt{T} {L}.
		\end{equation}
		We next use~\cref{assumption:A4} to obtain
		\begin{equation}
			\sum_{t=1}^T \left(\frac{1}{\alpha_t} - \frac{1}{\alpha_{t-1}}\right) \varphi_{1/\hat\rho}(\param_t) \leq \sum_{t=1}^T \left| \frac{1}{\alpha_t} - \frac{1}{\alpha_{t-1}}\right| |\varphi_{1/\hat\rho}(\param_t)|
			\leq  C_{\varphi}\sum_{t=1}^T  \left(\frac{1}{\alpha_t} - \frac{1}{\alpha_{t-1}}\right)
			\leq   C_{\varphi}\frac{\sqrt{v_0+ TL^2}}{\alpha}.\label{eq: ig4}
		\end{equation}
		since $\frac{1}{\alpha_t} = \frac{\sqrt{v_0+\sum_{j=1}^t \| G(\param_j, \x_{j+1})\|^2}}{\alpha}$ is monotonically nondecreasing in $t$.
		
		It remains to estimate the last term on~\eqref{eq: wie3} which is the main additional error term that is due to dependent data.
		For convenience, let us define $\alpha_s \| G(\param_s, \x_{s+1})\|=0$ for $s\leq 0$. Then we have 
		\begin{equation*}
			\sum_{t=1}^T \sum_{s=t-k_t}^{t-1} \alpha_s \|G(\param_{s}, x_{s+1})\| \leq \sum_{t=1}^T \sum_{s=t-k_T}^{t-1} \alpha_s \|G(\param_{s}, x_{s+1})\|,
		\end{equation*}
		where $\alpha_s = \frac{\alpha}{\sqrt{v_0+\sum_{j=1}^s \| G(\param_j, \x_{j+1}) \|^2}}$ and the inequality used that $k_t$ is nondecreasing.
		
		By Young's inequality, we can upper bound this term as
		\begin{align}
			\sum_{t=1}^T \sum_{s=t-k_T}^{t-1} \alpha_s \|G(\param_{s}, x_{s+1})\| &= \sum_{t=1}^T \left(\frac{(k_T)^{1/2}}{t^{1/4}} \right) \left( \frac{t^{1/4}}{(k_T)^{1/2}}\sum_{s=t-k_T}^{t-1} \alpha_s \|G(\param_{s}, x_{s+1})\|\right)  \\
			&\leq \sum_{t=1}^T \frac{k_T+1}{2\sqrt{t}} + \sum_{t=1}^T \frac{\sqrt{t}}{2k_T}\left(\sum_{s=t-k_T}^{t-1} \alpha_s \| G(\param_{s}, \x_{s+1}) \| \right)^2 \\
			&\leq \sqrt{T}k_T + \sum_{t=1}^T \frac{\sqrt{t}}{2k_T}\left(\sum_{s=t-k_T}^{t-1} \alpha_s \| G(\param_{s}, \x_{s+1}) \| \right)^2. \label{eq: wjk3}
		\end{align}
		We continue estimating the last term on RHS. Using the inequality $(\sum_{i=1}^{m}a_{i})^{2}\le m\sum_{i=1}^{m} a_{i}^{2}$ that follows from Cauchy-Schwarz, we get 
		\begin{align}
			\sum_{t=1}^T \frac{\sqrt{t}}{2k_T} \left( \sum_{s=t-k_T}^{t-1} \alpha_s \| G(\param_s, \x_{s+1})\| \right)^2 &\leq \sum_{t=1}^T \frac{\sqrt{t}}{2}  \sum_{s=t-k_T}^{t-1} \alpha_s^2 \| G(\param_s, \x_{s+1})\|^2\notag \\
			&\leq \frac{\sqrt{T}}{2} \sum_{t=1}^T   \sum_{s=t-k_T}^{t-1} \alpha_s^2 \| G(\param_s, \x_{s+1})\|^2 \notag \\
			&= \frac{\sqrt{T}}{2} \sum_{s=1}^{k_T}   \sum_{t=1}^{T-s} \alpha_t^2 \| G(\param_t, \x_{t+1})\|^2,
			\label{eq: fir3}
		\end{align}
		since for any $(c_s)$, we have $\sum_{t=1}^T \sum_{s=t-k_T}^{t-1} c_s = (c_{1-k_T} + c_{2-k_T} + \dots + c_0) + (c_{2-k_T} + c_{3-k_T} + \dots + c_{1}) +\dots+ (c_{T-k_T} + c_{T-k_T+1} + \dots + c_{T-1}) = (c_{1-k_T} + c_{2-k_T} + \dots + c_{T-k_T}) + (c_{2-k_T} + c_{3-k_T} + \dots + c_{T-k_T+1}) + \dots + (c_0 + c_{1} + \dots + c_{T-1}) = \sum_{s=1}^{k_T} \sum_{t=1-s}^{T-s} c_t$. Since in our case $c_t = 0$ for $t < 1$, we have also that $\sum_{s=1}^{k_T} \sum_{t=1-s}^{T-s} c_t = \sum_{s=1}^{k_T} \sum_{t=1}^{T-s} c_t$.
		
		We now have that the rightmost summation in~\eqref{eq: fir3} is of the form in the first inequality in Lem.~\ref{lem: num_seq}. 
		We continue from \eqref{eq: fir3} by using the definition of $\alpha_t$
		\begin{align*}
			\sum_{t=1}^T \frac{\sqrt{t}}{2k_T} \Big( \sum_{s=t-k_T}^{t-1} \alpha_s \| G(\param_s, &\x_{s+1})\| \Big)^2 \leq \frac{\sqrt{T}}{2} \sum_{s=1}^{k_T}   \sum_{t=1}^{T-s} \alpha_t^2 \| G(\param_t, \x_{t+1})\|^2 \\
			&= \frac{\sqrt{T}}{2} \sum_{s=1}^{k_T}   \sum_{t=1}^{T-s} \frac{\alpha^2}{v_0+\sum_{i=1}^t \| G(\param_i, \x_{i+1})\|^2} \| G(\param_t, \x_{t+1})\|^2 \\
			&\leq \frac{\sqrt{T}\alpha^2}{2} \sum_{s=1}^{k_{T}}  \log\left( 1+ v_0^{-1}\sum_{t=1}^{T-s} \| G(\param_t, \x_{t+1})\|^2 \right) \\
			&\leq \frac{\sqrt{T}k_T\alpha^2}{2}  \log\left( 1+ v_0^{-1}\sum_{t=1}^{T} \| G(\param_t, \x_{t+1})\|^2 \right),
		\end{align*}
		where the third line applies the second inequality in Lem.~\ref{lem: num_seq}.
		Using this estimation on~\eqref{eq: wjk3} gives us
		\begin{equation}
			\sum_{t=1}^T \sum_{s=t-k_T}^{t-1} \alpha_s \|G(\param_{s}, x_{s+1})\|\leq \sqrt{T}k_T + \frac{\sqrt{T}k_T\alpha^2}{2} \log\left( 1+ v_0^{-1}\sum_{t=1}^{T} \| G(\param_t, \x_{t+1})\|^2 \right).\label{eq: ig5}
		\end{equation}
		Collecting~\eqref{eq: ig3},~\eqref{eq: ig4} and~\eqref{eq: ig5} on~\eqref{eq: wie3} results in the bound
		\begin{multline}
			\frac{\hat\rho-\rho}{\hat\rho} \sum_{t=1}^{T}\mathbb{E}\left[ \| \nabla \varphi_{1/\hat\rho}(\param_t) \|^2 \right]\leq \frac{\sqrt{v_0+TL^2} C_{\varphi}}{\alpha} + \hat\rho L \sqrt{T}
			+ \sum_{t=1}^{T}\hat\rho \mathbb{E}\left[ \frac{4 L^2}{\hat\rho-\rho} \, \Delta_{[t-k,t]} \right] \\
			+ 2L\hat\rho \frac{L_1+\hat\rho }{\hat\rho-\rho} \left(\sqrt{T}k_T + \frac{\sqrt{T}k_T\alpha^2}{2} \log\left( 1+ v_0^{-1}\sum_{t=1}^{T} \| G(\param_t, \x_{t+1})\|^2 \right) \right).
		\end{multline}
		We divide both sides by $T$ to conclude.
	\end{proof}

	\section{Stochastic Heavy Ball with Dependent Data}\label{sec: stoc_hb}
	In this section, we focus on stochastic heavy ball method (Algorithm \ref{algorithm:shb}), a popular SGD method with momentum, which dates back to~\citep{polyak1964some}.
	This method is analyzed for convex optimization in~\citep{ghadimi2015global} and for constrained and stochastic nonconvex optimization with i.i.d. data in~\citep{mai2020convergence}.
	Some features of our analysis simplify and relax some conditions from the analysis in~\citep{mai2020convergence} even with i.i.d. data, see Lem.~\ref{lem: ekr4} and Remark~\ref{rem: shb} for the details.
	\begin{algorithm}[H]
		\small
		\caption{Stochastic heavy ball (momentum SGD)}
		\label{algorithm:shb}
		\begin{algorithmic}[1]
			\STATE \textbf{Input:} Initialize $\param_{0}\in \Param \subseteq \R^{p}$; $T > 0$;\,   $(\alpha_{t})_{t\ge 1}$;\, $\beta > 0$;\, $z_1 > 0$
			\STATE Optionally, sample $\tau$ from  $\{1,\dots,T\}$ independently of everything else where $\P(\tau=k)=\frac{\alpha_k}{\sum_{t=1}^T \alpha_t}$. 
			\STATE \textbf{For $t=1,2,\dots,T$ do:}
			\STATE \qquad Sample $\x_{t+1}$ from the conditional distribution $\pi_{t+1}=\pi_{t+1}(\cdot\,|\, \x_{1},\dots,\x_{t})$
			\STATE \qquad $\param_{t+1} \leftarrow \proj_{\Param} \left(  \param_{t} - \alpha_t z_t \right)$
			\STATE \qquad $z_{t+1} = \beta G(\param_{t+1}, \x_{t+1}) + \frac{1-\beta}{\alpha_{t+1}} \left( \param_t - \param_{t+1} \right)$
			\STATE \textbf{End for}
			\STATE \textbf{Return:} $\param_{T}$ (Optionally, return $\param_{\tau}$)
		\end{algorithmic}
	\end{algorithm}
	We start with a lemma showing a bound on the norm of the sequence $(z_k)$. We use this lemma to simplify some of the estimations in~\citep{mai2020convergence} that analyzed the algorithm in the i.i.d. case.
	\begin{lemma}\label{lem: ekr4}
		Let $(z_t)$ be defined as Alg.~\ref{algorithm:shb} and let~\cref{assumption:A3} hold. Then, we have
		\begin{align}
			\lVert z_{t+1} \rVert^{2} \le \beta L + (1-\beta) (\alpha_{t}/\alpha_{t+1})^{2} \lVert z_{t} \rVert^{2} \quad \text{for all $t\ge 1$}
		\end{align}
		and 
		\begin{equation*}
			\sum_{t=1}^T \beta \alpha_t^2 \|z_t\|^2 \leq \alpha_1^2\| z_1 \|^2 + \beta L^2 \sum_{t=1}^T \alpha_{t+1}^2.
		\end{equation*}
	\end{lemma}
	\begin{proof}
		By the definition of $z_t$ and convexity of $\| \cdot \|^2$, we have
		\begin{align}
			\| z_{t+1} \|^2 &\leq \beta \| G(\param_{t+1}, \x_{t+1})\|^2 + \frac{1-\beta}{\alpha_{t+1}^2} \| \param_t - \param_{t+1} \|^2 \\
			&\leq \beta \| G(\param_{t+1}, \x_{t+1})\|^2 + \frac{(1-\beta)\alpha_t^2}{\alpha_{t+1}^2}\| z_t \|^2,
		\end{align}
		where the second inequality used that $\param_t\in\Param$ and that $\proj_{\Param}$ is nonexpansive. Using \cref{assumption:A3} and dividing both sides by $\alpha_{t+1}^{2}$ gives the first inequality in the assertion. Also, by multiplying both sides of the inequality by $\alpha_{t+1}^2$, we have
		\begin{equation}\label{eq: eur4}
			\alpha_{t+1}^2 \| z_{t+1} \|^2 \leq \beta\alpha_{t+1}^2 \| G(\param_t, \x_{t+1})\|^2 + (1-\beta)\alpha_t^2 \| z_t \|^2.
		\end{equation}
		By using~\cref{assumption:A3} in~\eqref{eq: eur4}, then rearranging, multiplying both sides by $t^{\delta}$, and summing~\eqref{eq: eur4} give
		\begin{equation*}
			\sum_{t=1}^T \beta t^{\delta} \alpha_t^2 \| z_t \|^2 \leq - T^{\delta}\alpha_{T+1}^2\| z_{T+1} \|^2 + \alpha_1^2 \| z_1 \|^2 + \beta L^2\sum_{t=1}^T  t^{\delta}\alpha_{t+1}^2.
		\end{equation*}
		Removing the nonpositive term on the RHS gives the result.
	\end{proof}

  	\begin{theorem}[extended version of Theorem~\ref{th: th_shb} in the main text]\label{th: th_shb_appendix}
		Let \cref{assumption:A2}-\cref{assumption:A3} hold. Let $(\param_{t})_{t\ge 1}$ be a sequence generated by Algorithm \ref{algorithm:shb}. 
		Fix $\hat{\rho}\geq 2\rho$. 
		Then, for any $\beta \in (0, 1]$,
		\begin{description}
			\item[(i)] For each $T\ge 1$:
			{\small
				\begin{multline}
					\E\left[     \lVert \nabla \varphi_{1/\hat{\rho}}(\bar\param_{T}^{\mathrm{out}}) \rVert^{2} \right]   \le \frac{\hat\rho}{\sum_{t=1}^T \alpha_t } \Bigg( \varphi_{1/\hat\rho}(\bar{\param}_0) - \inf \varphi_{1/\hat\rho} + \frac{(1+\beta(1-\beta))L^2}{2\beta^2} \sum_{t=1}^T \alpha_t^2 \\ 
					+ \frac{1-\beta}{2\beta^2}\alpha_1 \| z_1\|^2
					+ \frac{ \commHL{4} L^2}{\hat\rho - \rho} \sum_{t=1}^T\alpha_t \mathbb{E}[\Delta_{[t-k_t, t]}] + \frac{2L^2(L_1 + \hat\rho)}{\hat\rho-\rho} \sum_{t=1}^T k_t\alpha_t\alpha_{t-k} \Bigg).
				\end{multline}
			}
			
			\item[(ii)] (Global convergence) Further assume that $\alpha_{t}/\alpha_{t+1}\rightarrow 1$ as $t\rightarrow \infty$ and $\sum_{t=1}^\infty k_t\alpha_t\alpha_{t-k_t} < \infty$. Then $\lVert \nabla \varphi_{1/\hat{\rho}}(\hat{\param}_{t}) \rVert \rightarrow 0$ as $t\rightarrow \infty$ almost surely. Furthermore, $\param_{t}$ converges to the set of all stationary points of $f$ over $\Param$. 
		\end{description}
	\end{theorem}
 
	\begin{remark}\label{rem: shb}
		Our analysis is more flexible compared to~\citep{mai2020convergence} even when restricted to the i.i.d. case.
		In this case, we allow variable step sizes $\alpha_t=\frac{\gamma}{\sqrt{t}}$ whereas~\citep{mai2020convergence} requires constant step size $\alpha_t=\alpha = \frac{\gamma}{\sqrt{T}}$.
		We can also use any $\beta\in(0, 1]$ whereas~\citep{mai2020convergence} restricts to $\beta=\alpha$. 
		This point is important since in practice $\beta$ is used as a tuning parameter.
	\end{remark}
	\begin{proof}
		We proceed as the proof of~Thm.~\ref{th: th_proj_sgd}. However, following the existing analyses for SHB~\citep{ghadimi2015global,mai2020convergence} we use the following iterate $\bar{\param}_t = \param_t + \frac{1-\beta}{\beta}\left( \param_t - \param_{t-1} \right)$ and also $\hat{\param}_t = \prox_{\varphi/\hat{\rho}}(\bar{\param}_t)$.
		The useful property of $\bar{\param}_t$ exploited in~\citep{mai2020convergence} with constant step sizes (and also in~\citep{ghadimi2015global} in the unconstrained setting), is that
		\begin{align}
			\| \bar{\param}_{t+1} - \hat{\param}_t \|^2 &= \left\| \param_{t+1} + \frac{1-\beta}{\beta}(\param_{t+1} - \param_t) - \hat{\param}_t \right\|^2 = \frac{1}{\beta^2} \| \param_{t+1} - [(1-\beta) \param_t + \beta \hat{\param}_t] \|^2 \\
			&\leq \frac{1}{\beta^2} \| \param_t - \alpha_t z_t - [(1-\beta) \param_t + \beta \hat{\param}_t] \|^2 \\
			&= \| \bar{\param}_t -  \hat{\param}_t - \alpha_t G(\param_t, \x_{t}) \|^2,\label{eq: rht4}
		\end{align}
		where the inequality used that $\param_t, \param_{t+1}, \hat{\param}_t$ and their convex combinations are feasible points and the projection is nonexpansive.
		The last step is by simple rearrangement and using the definition of $z_t$.
		
		On the first chain of inequalities in~Thm.~\ref{th: th_proj_sgd}, we evaluate $\varphi_{1/\hat\rho}$ at $\bar{\param}_{t+1}$ instead of $\param_{t+1}$ and then use the inequality in~\eqref{eq: rht4} to deduce
		\begin{align}
			\E\left[ \varphi_{1/\hat{\rho}} (\bar{\param}_{t+1}) \,\bigg|\, \mathcal{F}_{t-k} \right] &\le \E\left[ f(\hat{\param}_{t}) + \frac{\hat{\rho}}{2} \lVert \bar{\param}_{t+1} - \hat{\param}_{t}\rVert^{2} \,\bigg|\, \mathcal{F}_{t-k} \right] \\
			&\leq \E\left[ f(\hat{\param}_{t}) + \frac{\hat{\rho}}{2} \lVert \bar{\param}_{t} - \alpha_t G(\param_t, \x_t) - \hat{\param}_{t}\rVert^{2} \,\bigg|\, \mathcal{F}_{t-k} \right].\label{eq: fir4}
		\end{align}
		We expand the square to obtain
		\begin{align}
			\left\lVert \bar{\param}_{t} - \alpha_t G(\param_t, \x_t) - \hat{\param}_{t}\right\rVert^{2} = \| \bar{\param}_t - \hat{\param}_t \|^2 - 2\alpha_t \langle \bar{\param}_t - \hat{\param}_t, G(\param_t, \x_t) \rangle + \alpha_t^2 \| G(\param_t, \x_t)\|^2.
		\end{align}
		By using the last estimate on~\eqref{eq: fir4} and using the definition of $\varphi_{1/\hat\rho}$, $\bar{\param}_t$ along with~\cref{assumption:A3} gives
		\begin{multline}
			\E\left[ \varphi_{1/\hat{\rho}} (\bar{\param}_{t+1}) \,\bigg|\, \mathcal{F}_{t-k} \right] \le \E\bigg[ \varphi_{1/\hat{\rho}} (\bar{\param}_{t})  - \hat\rho\alpha_t \langle {\param}_t - \hat{\param}_t, G(\param_t, \x_t) \rangle \\
			- \frac{\hat\rho\alpha_t(1-\beta)}{\beta} \langle {\param}_t - {\param}_{t-1}, G(\param_t, \x_t) \rangle  + \frac{\hat\rho\alpha_t^2 L^2}{2} \,\bigg|\, \mathcal{F}_{t-k} \bigg].\label{eq: fir5}
		\end{multline}
		We estimate the third term on RHS by Young's inequality, the nonexpansiveness of the projection and~\cref{assumption:A3}
		\begin{align}
			- \frac{\hat\rho\alpha_t(1-\beta)}{\beta} \langle {\param}_t - {\param}_{t-1}, G(\param_t, \x_t) \rangle &\leq \frac{\hat\rho(1-\beta)}{2\beta} \left( \| \param_t - \param_{t-1} \|^2 + \alpha_t^2 \| G(\param_t, \x_t) \|^2 \right) \\
			&\leq \frac{\hat\rho(1-\beta)}{2\beta} \left( \alpha_{t-1}^2 \| z_{t-1} \|^2 + \alpha_t^2 L^2 \right).
		\end{align}
		We insert this estimate back to~\eqref{eq: fir5} and use Lem.~\ref{lem:TV_linear_change_bound} as in the proof of Thm.~\ref{th: th_proj_sgd} to obtain
		\begin{align}
			\E\Bigg[ \varphi_{1/\hat{\rho}} (\bar{\param}_{t+1}) &\,\bigg|\, \mathcal{F}_{t-k} \Bigg] \le \E\bigg[ \varphi_{1/\hat{\rho}} (\bar{\param}_{t+1})  - \hat\rho\alpha_t \langle {\param}_t - \hat{\param}_t, G(\param_t, \x_t) \rangle + \frac{\hat\rho(1-\beta)}{2\beta}\alpha_{t-1}^2 \|z_t\|^2 \\
			&\hspace{6cm}+ \frac{\hat\rho(2-\beta)\alpha_t^2 L^2}{2\beta} \,\bigg|\, \mathcal{F}_{t-k} \bigg] \notag \\
			&\le \E\left[ \varphi_{1/\hat{\rho}}(\param_{t}) \,\bigg|\, \mathcal{F}_{t-k}\right] - \hat{\rho} \alpha_{t}  \langle\param_t- \hat{\param}_{t} ,\, \E_{\x\sim \pi}\left[ G(\param_{t}, \x) \right] \rangle  + \frac{\hat\rho(2-\beta)\alpha_{t}^{2} L^{2} }{2\beta} \notag \\
			&\qquad + \hat\rho\alpha_t \left( \frac{2L^2}{\hat\rho-\rho} \, \Delta_{[t-k,t]}  + k\frac{2L^2 L_1+\hat\rho L^2}{\hat\rho-\rho} \alpha_{t-k}\right) \\
			&\hspace{5cm}+ \frac{\hat\rho(1-\beta)}{2\beta} \alpha_{t-1}^2 \mathbb{E}\left[\| z_{t-1}\|^2 \,\bigg|\, \mathcal{F}_{t-k} \right].\label{eq: rie3}
		\end{align}
		We now estimate the second term on the RHS
		\begin{align}
			\langle \param_t - \hat{\param}_t&, \mathbb{E}_{\x\sim \pi}[G(\param_t, x)] \rangle \geq f(\param_t) - f(\hat{\param}_t) - \frac{\rho}{2} \| \param_t - \hat{\param}_t\|^2 \\
			&= \left( f(\param_t) + \frac{\hat\rho}{2} \| \param_t - \bar{\param}_t \|^2 \right) - \left( f(\hat{\param}_t) + \frac{\hat\rho}{2} \| \hat{\param}_t - \bar{\param}_t \|^2 \right) - \frac{\hat\rho}{2} \| \param_t - \bar{\param}_t \|^2 \\
			&\hspace{3cm}+ \frac{\hat\rho}{2} \| \hat{\param}_t - \bar{\param}_t\|^2 - \frac{\rho}{2} \| \param_t - \hat{\param}_t\|^2 \\
			&\geq \frac{\hat\rho}{2} \| \hat{\param}_t - \bar{\param}_t \|^2 - \frac{\hat\rho}{2} \| \param_t - \bar{\param}_t \|^2\geq \frac{\hat\rho}{2} \| \hat{\param}_t - \bar{\param}_t \|^2 - \frac{\hat\rho(1-\beta)^2\alpha_{t-1}^2}{2\beta^2} \| z_{t-1} \|^2.\label{eq: tir3}
		\end{align}
		where the first inequality is due to $\rho$-weak convexity of $f$, and the second inequality is by $\hat\rho-\rho$-strong convexity of $f(\cdot) + \frac{\hat\rho}{2} \| \cdot - \bar{\param}_t\|^2$ with the optimizer $\hat{\param}_t$ and $\hat\rho \geq 2\rho$.
		The third inequality is by nonexpansiveness of the projection and the definition of $\bar{\param}_t$.
		
		We use~\eqref{eq: tir3} on~\eqref{eq: rie3}, insert $k=k_t$, integrate out $\mathcal{F}_{t-k}$ and sum to get
		\begin{multline}\label{eq:SHB_infinite_sum}
			\sum_{t=1}^T \hat\rho^2\alpha_t\mathbb{E} \left[  \| \bar{\param}_t - \hat{\param}_t \|^2 \right] \leq -\E\left[ \varphi_{1/\hat{\rho}} (\bar{\param}_{T+1}) \right] +   \varphi_{1/\hat{\rho}}(\param_{1}) +\sum_{t=1}^T\frac{\hat{\rho}(2-\beta)\alpha_{t}^{2}  L^{2} }{2\beta} \\
			+ \sum_{t=1}^T \frac{\hat\rho(1-\beta)^2\alpha_{t-1}^2}{2\beta^2}  \mathbb{E}\| z_{t-1} \|^2
			+\sum_{t=1}^T \hat\rho\alpha_t\left( \frac{4 L^2}{\hat\rho-\rho} \, \mathbb{E}[\Delta_{[t-k_t,t]}]  + k_t\frac{2L^2 L_1+\hat\rho L^2}{\hat\rho-\rho} \alpha_{t-k_t}\right) \\
			+\sum_{t=1}^T \frac{\hat\rho(1-\beta)}{2\beta} \alpha_{t-1}^2 \mathbb{E}\left[\| z_{t-1}\|^2 \right].
		\end{multline}
		Using Lem.~\ref{lem: ekr4} for the terms involving $\|z_t\|^2$ and using $\| \nabla \varphi_{1/\hat\rho}(\bar{\param}_t) \| = \hat\rho\| \hat{\param}_t - \bar{\param}_t \|$ finishes the proof of \textbf{(i)} after simple arrangements.

		Next, we show \textbf{(ii)}. 
		The argument for the second part is identical to that of Theorem \ref{th: th_shb} \textbf{(ii)}. The argument for the first part is also similar to that of Theorem \ref{th: th_proj_sgd} \textbf{(ii)} with a minor modification. Namely, from \eqref{eq:SHB_infinite_sum} and the hypothesis,
		\begin{align}
			\sum_{t=1}^T \hat\rho^2\alpha_t\mathbb{E} \left[  \| \bar{\param}_t - \hat{\param}_t \|^2 \right] < \infty. 
		\end{align}
		Using Fubini's theorem and \eqref{eq:moreau_grad}, this implies 
		\begin{align}
			\sum_{t=1}^T \alpha_t  \| \nabla \varphi_{1/\hat{\rho}}(\bar{\param}_{t}) \|^2  < \infty \qquad \textup{almost surely}. 
		\end{align}
		Hence by Lemma \ref{lem:positive_convergence_lemma}, it suffices to show that 
		\begin{align}
			\left|  \| \nabla \varphi_{1/\hat{\rho}}(\bar{\param}_{t+1}) \|^2  - \| \nabla \varphi_{1/\hat{\rho}}(\bar{\param}_{t}) \|^2 \right| =O(\alpha_{t}).
		\end{align}
		Proceeding as in the proof of Theorem \ref{th: th_proj_sgd} \textbf{(ii)}, the above follows if $\lVert z_{t} \rVert$ is uniformly bounded. 
		
		It remains to show that $\lVert z_{t} \rVert$ is uniformly bounded. For this, it suffices to show that $\lVert z_{t} \rVert^{2} \le 2L $ for all sufficiently large $t\ge 1$. We deduce this from Lemma \ref{lem: ekr4}. If $\beta=1$, the lemma implies $\lVert z_{t}\rVert^{2}\le L$ for all $t\ge 1$, so we may assume $\beta<1$. Proceeding by an induction on $t$, suppose this bound holds for $z_{t}$. Then by Lemma \ref{lem: ekr4}, we have 
		\begin{align}
			\lVert z_{t+1} \rVert^{2} \le  \beta L + 2(1-\beta)  (\alpha_{t}/\alpha_{t+1})^{2}L.
		\end{align}
		Since $\beta<1$ and $\alpha_{t}/\alpha_{t+1}\rightarrow 1$ as $t\rightarrow \infty$, there exists $t_{0}>0$ such that for all $t>t_{0}$, $(1-\beta)  (\alpha_{t}/\alpha_{t+1})^{2}<1-\beta/2$. Therefore, for all $t>t_{0}$, 
		\begin{align}
			\lVert z_{t+1} \rVert^{2} \le \beta L + (1-\beta/2)(2L) = 2L.
		\end{align}
		This shows the assertion. 
	\end{proof}

	\section{Proximal SGD with Dependent Data}\label{sec: prox_extension}
	In this section, we describe how our developments for stochastic gradient method extends to the proximal case, using the ideas from~\citep{davis2019stochastic}.
	In particular, the problem we solve in this section is
	\begin{align}
		\param^{*} \in \argmin_{\param\in \R^{p}} \left( \varphi(\param):= f(\param) + r(\param)\right), \quad f(\param)= \E_{\x\sim \pi}\left[ \ell(\param, \x)  \right],
	\end{align} 
	where $r\colon \mathbb{R}^p \to \mathbb{R} \cup {\{ +\infty \}}$ is a convex, proper, closed function.
	In this case, in step~\ref{eq: woe3} of Algorithm \ref{algorithm:prox_grad}, we use $\prox_{\alpha_t r}$ instead of $\proj_{\Param}$ to define $\param_{t+1}$.

	Recall also that
	\begin{equation*}
		\hat{\param}_t=\prox_{\varphi/\hat{\rho}}(\param_t).
	\end{equation*}
	In the projected case, when $r(\param)$ is the indicator function of the set $\Param$, we had that $\hat{\param}_t \in \Param$.
	This was used, for example, in~\eqref{eq: wej4} to use nonexpansiveness for bounding $\| \param_{t+1} - \hat{\param}_t \|^2$ since $\param_{t+1} = \proj_{\Param}(\param_t - \alpha_t g_t)$.
	In this case, for the same step, one needs an intermediate result derived by~\citep{davis2019stochastic}.
	\begin{lemma}\cite{davis2019stochastic}
		Given the definition of $\hat{\param}_t$, we have for $t\geq 0$
		\begin{align*}
			\hat{\param}_t = \prox_{\alpha_t r}(\alpha_t\hat\rho \param_t - \alpha_t \hat v_t + (1-\alpha_t\hat\rho)\hat{\param}_t),
		\end{align*}
		where $\hat v_t \in \partial f(\hat{\param}_t)$.
	\end{lemma}
	We include the following result combining the ideas from Lem.~\ref{lem:TV_linear_change_bound}, Thm.~\ref{th: th_proj_sgd} and \cite{davis2019stochastic} for proving convergence of proximal stochastic gradient algorithm with dependent data.

 
	\begin{theorem}\label{th:prox_appendix}[Theorem~\ref{th:prox} in the main text]
		Let \cref{assumption:A2}-\ref{assumption:A3} hold, $r$ be convex, proper, closed and $(\param_{t})_{t\ge 1}$ be a sequence generated by Algorithm \ref{algorithm:prox_grad} where we use $\prox_{\alpha_t r}$ instead of $\proj_{\Param}$ in step~\ref{eq: woe3}.
		Fix $\hat{\rho}>\rho$. For each $T\ge 1$, 
		\begin{align}
			&\E\left[     \lVert \nabla \varphi_{1/\hat{\rho}}(\param_{T}^{\mathrm{out}}) \rVert^{2} \right]  \\
			&\quad \le \frac{\hat{\rho}^{2}L^2}{\hat{\rho}-\rho} \frac{1}{\sum_{k=1}^{T} \alpha_{k}}\Big[\frac{ \varphi_{1/\hat{\rho}}(\param_{0}) - \inf \varphi_{1/\hat{\rho}}}{\hat{\rho}L^2}   + 2 \sum_{t=1}^{T} \alpha_{t}^{2} + \frac{2(L_1 + \hat\rho)}{\hat\rho-\rho} \sum_{t=1}^{T} k_{t}\alpha_{t}\alpha_{t-k_{t}} \\
			&\quad+ \frac{\commHL{4}}{\hat\rho-\rho} \sum_{t=1}^{T} \alpha_{t} \mathbb{E}[\Delta_{[t-k_{t}, t]}]\Big].
		\end{align}
	\end{theorem}
	\begin{proof}
		We start the same as~Thm.~\ref{th: th_proj_sgd} and note by the definition of $\hat{\param}_{t+1}$
		\begin{align}
			\varphi_{1/\hat\rho}(\param_{t+1}) \leq \varphi(\hat{\param}_t) + \frac{\hat\rho}{2} \|   \param_{t+1} - \hat{\param}_t \|^2.\label{eq: tur3}
		\end{align}
		We next estimate $\frac{\hat\rho}{2} \|   \param_{t+1} - \hat{\param}_t \|^2$ similar to~\cite{davis2019stochastic} by using $1$-Lipschitzness of $\prox_{\alpha_t g}$. Let $\delta = 1-\alpha_t \hat\rho$ and estimate
		\begin{align}
			\| \param_{t+1} - \hat{\param}_t \|^2 &= \| \prox_{\alpha_t r}(\param_t - \alpha_t g_t) - \prox_{\alpha_t r}(\alpha_t\hat\rho \param_t - \alpha_t \hat v_t + \delta\hat{\param}_t) \|^2 \\
			&\leq \delta^2 \| \param_t - \hat{\param}_t \|^2 - 2\delta\alpha_t \langle \param_t - \hat{\param}_t, G(\param_t, \x_{t}) - \hat v_t \rangle + \alpha_t^2 \| G(\param_t, \x_t) - \hat v_t \|^2,\label{eq: tur3.5}
		\end{align}
		where we skipped some intermediate steps, which are already in~\cite{davis2019stochastic}.
		We note that by Lem.~\ref{lem:TV_linear_change_bound}, we have
		\begin{multline}
			-2\delta\alpha_t \mathbb{E}\left[\langle \param_t - \hat{\param}_t, G(\param_t, \x_{t}) \rangle \,\bigg|\, \mathcal{F}_{t-k} \right] = -2\delta\alpha_t \langle \param_t - \hat{\param}_t, \mathbb{E}_{\x\sim \pi}[G(\param_t, \x_{t})] \rangle + \\
			+2\delta\alpha_t \left( \frac{2L^2}{\hat\rho-\rho} \Delta_{[t-k, t]} + k\frac{2L^2( L_1 + \hat\rho)}{\hat\rho-\rho}\alpha_{t-k} \right).\label{eq: tur4}
		\end{multline}
		We take the conditional expectation of~\eqref{eq: tur3} and use~\eqref{eq: tur3.5} with~\eqref{eq: tur4} to derive
		\begin{multline}
			\E\left[\varphi_{1/\hat\rho}(\param_{t+1}) \, \bigg| \, \mathcal{F}_{t-k}\right] \leq \ \E\left[\varphi(\hat{\param}_t)\, \bigg| \, \mathcal{F}_{t-k}\right] +\delta^2 \| \param_t - \hat{\param}_t \|^2 \\
			-2\delta\alpha_t \langle \param_t - \hat{\param}_t, \mathbb{E}_{\x\sim \pi}[G(\param_t, \x_{t})]\rangle -2\delta\alpha_t \mathbb{E}\left[\langle \param_t - \hat{\param}_t, -\hat v_t \rangle \, \bigg| \, \mathcal{F}_{t-k}\right] \\
			+ \alpha_t^2 \mathbb{E}\left[ \| G(\param_t, \x_{t+1}) - \hat v_t \|^2 \, \bigg| \, \mathcal{F}_{t-k} \right] +2\delta\alpha_t \left( \frac{{4}L^2}{\hat\rho-\rho} \Delta_{[t-k, t]} + k\frac{2L^2( L_1 + \hat\rho )}{\hat\rho-\rho}\alpha_{t-k} \right)
		\end{multline}
		We integrate out $\mathcal{F}_{t-k}$ to obtain
		\begin{multline}
			\E\left[\varphi_{1/\hat\rho}(\param_{t+1}) \right] \leq \ \E\left[\varphi(\hat{\param}_t)\right]+\delta^2 \| \param_t - \hat{\param}_t \|^2 -2\delta\alpha_t \E\left[\langle \param_t - \hat{\param}_t, \mathbb{E}_{\x\sim \pi}[G(\param_t, \x_{t})] - \hat v_t\rangle \right] \\
			+ \alpha_t^2 \mathbb{E}\left[ \| G(\param_t, \x_{t+1}) - \hat v_t \|^2 \right] +2\delta\alpha_t \left( \frac{{4}L^2}{\hat\rho-\rho} \mathbb{E}[\Delta_{[t-k, t]}] + 4k\frac{2L^2 L_1 + \hat\rho L^2}{\hat\rho-\rho}\alpha_{t-k} \right).\label{eq: riw2}
		\end{multline}
		Next, we use that the subdifferential of $\rho$-weakly convex $g$ is $\rho$-hypomonotone (see~\cite{davis2019stochastic}) and $\mathbb{E}_{\x\sim \pi}[G(\param_t, \x)] \in \partial f(\param_t)$ and $\hat v_t \in \partial f(\hat{\param}_t)$ to derive
		\begin{equation}\label{eq: riw3}
			\langle \param_t - \hat{\param}_t, \mathbb{E}_{\x\sim \pi}[G(\param_t, \x_{t})] - \hat v_t\rangle \geq - \rho \| \param_t - \hat{\param_t} \|^2.
		\end{equation}
		We combine~\eqref{eq: riw3} with $\| \hat v_t \|^2 \leq L^2$ (see~\cite{davis2019stochastic}) on~\eqref{eq: riw2} to derive
		\begin{align}
			\E\left[\varphi_{1/\hat\rho}(\param_{t+1}) \right] &\leq \ \E\left[\varphi_{1/\varphi}({\param_t})\right] -\hat\rho(\hat\rho-\rho)\alpha_t \mathbb{E}\| \param_t - \hat{\param}_t \|^2
			+ 4\alpha_t^2L^2 \\
			&\qquad +2\delta\alpha_t \left( \frac{4 L^2}{\hat\rho-\rho} \mathbb{E}[\Delta_{[t-k, t]}] + k\frac{2L^2( L_1 + \hat\rho) }{\hat\rho-\rho}\alpha_{t-k} \right).
		\end{align}
		We sum the inequality and argue similarly as in the proof of Theorem \ref{th: th_proj_sgd} to finish the proof.
	\end{proof}
	
	\section{Proofs for Section~\ref{sec: smooth}}\label{app: smooth}

	\begin{lemma}\label{lem: post_process_output}
		Let Assumptions~\ref{assumption:A2},~\ref{assumption:A1},~\ref{assumption:A3} hold, $\Param$ be compact, and $\Delta_{[t- k_t, t]} = O(\lambda^{k_t})$ for $\lambda < 1$. 
		Let an algorithm output $\param_{t}$ (for example, a randomly selected iterate) such that $\mathbb{E}\| \param_t - \proj_{\Param}(\param_t - \nabla f(\param_t)) \| \leq \varepsilon$ with $\tilde{O}(\varepsilon^{-4})$ queries to $\nabla \ell(\param, \x)$. Then, for $\breve{\param}_{t+1} = \proj_{\Param}\left(\param_t - \tilde{\nabla}f(\param_t)) \right)$ with $\tilde{\nabla} f(\param_t) = \frac{1}{\hat N}\sum_{i=1}^{\hat N} \nabla \ell(\param_t, \x^{(i)})$ with $\hat N=O(\varepsilon^{-2})$ samples, we have that 
		\begin{equation*}
			\mathbb{E}\left[\mathrm{dist}({\mathbf{0}}, \partial(f+\iota_{\Param})(\breve{\param}_{t+1}))\right] \leq \varepsilon \text{~~~with~~~} \tilde{O}(\varepsilon^{-4}) \text{~~~samples.}
		\end{equation*}
	\end{lemma}    
	\begin{proof}[Proof of Lemma~\ref{lem: post_process_output}]
		By the definition of $\breve{\param}_{t+1}$, we have that
		\begin{equation*}
			\param_t - \tilde{\nabla} f(\param_t) -\breve{\param}_{t+1} \in \partial \iota_{\Param}(\breve{\param}_{t+1}).
		\end{equation*}
		As a result, we have
		\begin{align*}
			\mathbb{E}\left[\mathrm{dist}(\mathbf{0}, \partial(f+\iota_{\Param})(\breve{\param}_{t+1}))\right] &= \mathbb{E}\left[ \min_{v \in \partial \iota_{\Param}(\breve{\param}_{t+1})}\| \nabla f(\breve{\param}_{t+1}) + v \| \right]\\
			&\leq \mathbb{E} \| \nabla f(\breve{\param}_{t+1}) - \breve{\param}_{t+1} + \param_t - \tilde{\nabla} f(\param_t) \|.
		\end{align*}
		For convenience, let $\tilde{\param}_{t+1} = \proj_{\Param}(\param_t - \nabla f(\param_t))$.
		We continue estimating the last inequality by using this definition, triangle inequality, nonexpansiveness of $\proj_{\Param}$, and $\rho$-smoothness of $f$
		\begin{align*}
			\mathbb{E}\Big[\mathrm{dist}(\mathbf{0}, \partial(f+\iota_{\Param})&(\breve{\param}_{t+1}))\Big] \\
			&\leq \mathbb{E} \left[ \| \param_t - \breve{\param}_{t+1} \| + \| \nabla f(\breve{\param}_{t+1}) - \nabla f(\param_t)\| + \|\tilde{\nabla}f(\param_t)-\nabla f(\param_t) \| \right] \\
			&\leq \mathbb{E} \left[ (1+\rho)\| \param_t - \breve{\param}_{t+1} \| + \|\tilde{\nabla}f(\param_t)-\nabla f(\param_t) \| \right] \\
			&\leq \mathbb{E} \left[ (1+\rho)\left(\| \param_t - \tilde{\param}_{t+1} \| + \| \tilde{\param}_{t+1}- \breve{\param}_{t+1} \|\right) + \|\tilde{\nabla}f(\param_t)-\nabla f(\param_t) \| \right]\\
			&\leq \mathbb{E} \left[ (1+\rho)\| \param_t - \tilde{\param}_{t+1} \| + (2+\rho) \|\tilde{\nabla}f(\param_t)-\nabla f(\param_t) \| \right].
		\end{align*}
		By the assumption in the lemma, recall that we have $\|\param_t - \tilde{\param}_{t+1} \| \leq \varepsilon$, therefore we have to estimate the last term in the last inequality. We use Lem. 7.1 in \cite{lyu2022convergence} (see also Lemma \ref{lem:f_n_concentration_L1_gen}) with $\psi = \nabla \ell$ to get $\mathbb{E}\| \tilde{\nabla}f(\param_t) - \nabla f(\param_t) \| = O(\hat N^{-1/2})$ with $\hat N$ samples and finish the proof.
	\end{proof}
	
	\begin{proof}[Proof of Theorem~\ref{th: th_proj_sgd_smooth}]
		When $g$ is smooth, we can use the results in Sec. 2.2 in \cite{davis2019stochastic} to show that for any $\param$,
		\begin{align*}
			\|\mathcal{G}_{1/2\hat\rho}(\param)\| \leq \frac{3}{2} \| \nabla \varphi_{1/\hat\rho}(\param) \|.
		\end{align*}
		This establishes that the upper bound of Thm.~\ref{th: th_proj_sgd} also upper bounds the norm of the gradient mapping $\|\mathcal{G}_{1/2\hat\rho}(\param)\|$. By invoking Thm.~\ref{th: th_proj_sgd} with a randomly selected iterate, this establishes the bound required for Lem.~\ref{lem: post_process_output} and then applying Lem.~\ref{lem: post_process_output} gives the result.
	\end{proof}
	
	\section{Proof and discussions for Section~\ref{section:ODL}}\label{app: odl}
	\begin{proof}[Proof of Corollary~\ref{cor: odl}]
		Follows immediately from Theorems  \ref{th: th_proj_sgd}, \ref{th: th_adagrad}, \ref{th: th_shb} and \ref{th: th_proj_sgd_smooth}. For the last statement for squared Frobenius loss, see \cite{mairal2010online} for verifying  \cref{assumption:A5} and \cref{assumption:A6} and recall that \cref{assumption:A6} implies \cref{assumption:A2}.
	\end{proof}

	\section{Details about the experimental setup}
	
	For our experimental setup, we implemented the SGD based algorithms we have in this paper. The implementation of SMM uses one step of dictionary learning update given in~\cite{mairal2010online} with the special step size therein.
	We did not tune SMM further since the algorithm is well-established and specialized for ODL tasks, since the work of~\citep{mairal2010online}.
	
	For projected SGD and projected SGD with momentum, we used a step size of the form
	\begin{equation*}
		\alpha_t = \frac{c}{\sqrt{t+1}},
	\end{equation*}
	and tuned $c\in [0.01, 1]$. In~\cite{mairal2010online} and~\citep{zhao2017online}, the authors noted that using a step size $\alpha_t = \frac{c_1}{c_2 t + c_3}$ and tuning $c_1, c_2, c_3$ for SGD seemed to work well. We did not choose this rule in order not to tune three different parameters and since, as we show with our analysis, the best complexity is attained with a scaling of $\frac{1}{\sqrt{t}}$ for the step size.
	Consistent with~\citep{mairal2010online,zhao2017online}, we also observed further tuning with such a rule enhances the empirical performance of SGD-based methods. However we refrain from such a specialized tuning, since our goal is not to provide an exhaustive practical benchmark, but to enhance the theoretical understanding of algorithms whose practical merit is already well-established in a wide variety of tasks (SGD, SGD with momentum and AdaGrad). 
	
	For AdaGrad, we picked the step size
	\begin{equation*}
		\alpha_t = \frac{c}{\sqrt{\sum_{i=1}^t \| G(\param_t, \x_{t+1})\|^2}+\varepsilon},
	\end{equation*}
	with $\varepsilon = 10^{-8}$ and $c\in[0.1, 1]$ is tuned.

	\section{Convergence of PSGD in the state-dependent case}\label{sec: we3}

	\begin{theorem}[extended version of Theorem~\ref{th: th_proj_sgd_state_dep} in the main text]\label{thm:PSGD_state_dependent_MC}
		Let Assumptions~\ref{assumption:A2}, \ref{assumption:A1'}, \ref{assumption:A2'}, and  \ref{assumption:A3} hold. Let $(\param_{t})_{t\ge 1}$ be a sequence generated by Algorithm \ref{algorithm:prox_grad}. 
		Fix $\hat{\rho}>\rho$. Then we have for each $T\ge 1$ that
		\begin{align}
			\E\left[     \lVert \nabla \varphi_{1/\hat{\rho}}(\param_{T}^{\mathrm{out}}) \rVert^{2} \right] &\le \frac{\hat{\rho}^{2}L^2}{\hat{\rho}-\rho} \frac{1}{\sum_{k=1}^{T} \alpha_{k}}\bigg[\frac{ \varphi_{1/\hat{\rho}}(\param_{1}) - \inf \varphi_{1/\hat{\rho}}}{\hat{\rho}L^2}   + \frac{1}{2} \sum_{t=1}^{T} \alpha_{t}^{2}\bigg] \\
			&\qquad + \frac{\hat\rho}{\sum_{k=1}^{T} \alpha_{k}} \bigg[\frac{\alpha_1}{2} \left(\| \param_1 - \hat\param_1 \|^2 + C_1^2\right) + \frac{\alpha_T}{2} \frac{4LC_2}{\hat\rho - \rho} + \sum_{t=2}^T \frac{2L^2C_3\alpha_t\alpha_{t-1}}{\hat\rho-\rho} \\
			&\qquad+ \sum_{t=2}^T C_2 \alpha_t\left( \alpha_{t-1}L + \frac{\hat\rho}{\hat\rho-\rho}\alpha_{t-1} L \right) + \sum_{t=2}^T |\alpha_{t-1} - \alpha_t | \frac{2LC_2}{\hat\rho-\rho} \bigg].
		\end{align}
		In particular, with $\alpha_t = \frac{c}{\sqrt{t}}$ for some $c> 0$, we have that
		\begin{equation*}
			\E\left[     \lVert \nabla \varphi_{1/\hat{\rho}}(\param_{T}^{\mathrm{out}}) \rVert \right] \leq \varepsilon \text{~~with~~} \tilde{O}\left( \varepsilon^{-4} \right) \text{~~samples.}
		\end{equation*}
	\end{theorem}
	\begin{remark}
		Even though our main focus is operating under~\cref{assumption:A2} which is the main assumption on the data used in most of the other works we compare with~\citep{lyu2022convergence}, we also give this theorem for completeness. This theorem operates under another assumption depending on the solution of Poisson equation and is used in~\citep{karimi2019non,tadic2017asymptotic}. By using these techniques, we show that we can extend the guarantees in these papers to the constrained case. One difference is that in the constrained case, we need a slightly stronger assumption on the norms of the gradients, see~\cref{assumption:A3}.
	\end{remark}
	\begin{proof}
		We will follow the proof of Theorem~\ref{th: th_proj_sgd} until~\eqref{eq: htr3} which is where the main error term due to non-i.d.d. data appears. We rewrite this inequality for convenience, after taking total expectation and summing the inequality for $t\geq 1$
		\begin{align}
			\sum_{t=1}^T \mathbb{E}\left[\varphi_{1/\hat{\rho}} (\param_{t+1})\right] &\le  \sum_{t=1}^T \mathbb{E}\left[\varphi_{1/\hat{\rho}} (\param_{t})\rVert^{2}\right] + \sum_{t=1}^T \hat{\rho} \alpha_{t} \mathbb{E} \langle \hat{\param}_{t} - \param_{t},\, G(\param_{t}, \x_{t+1})\rangle  + \sum_{t=1}^T\frac{\alpha_{t}^{2} \hat{\rho} }{2} \mathbb{E}\| G(\param_t, \x_{t+1})\|^2 \notag\\
   &= \sum_{t=1}^T \mathbb{E}\left[\varphi_{1/\hat{\rho}} (\param_{t})\rVert^{2}\right] +\sum_{t=1}^T \hat{\rho} \alpha_{t} \mathbb{E} \langle \hat{\param}_{t} - \param_{t},\, G(\param_{t}, \x_{t+1})-\nabla f(\param_t)\rangle
      \notag\\
&\quad + \sum_{t=1}^T \hat{\rho} \alpha_{t} \mathbb{E} \langle \hat{\param}_{t} - \param_{t},\, \nabla f(\param_{t})\rangle+ \sum_{t=1}^T\frac{\alpha_{t}^{2} \hat{\rho} }{2} \mathbb{E}\| G(\param_t, \x_{t+1})\|^2   \label{eq: tfr3}
		\end{align}
		We have to then bound for the second term on the right-hand side:
		\begin{equation}\label{eq: ser1}
			\left\vert \mathbb{E}\hat \rho \sum_{t=1}^T \alpha_t \langle \hat{\param}_t - \param_t, \nabla f(\param_t) - G(\param_t, \x_{t+1}) \rangle\right\vert .
		\end{equation}
		We can then simply follow the same strategy as~\citep{karimi2019non} to obtain the result. For clarity, we write down these steps explicitly in the rest of this proof.
  
  In particular, by~\eqref{eq: gk5}, we have
		\begin{equation*}
			\hat \rho \sum_{t=1}^T \alpha_t \langle \hat{\param}_t - \param_t, \nabla f(\param_t) - G(\param_t, x_{t+1}) \rangle = -\hat \rho \sum_{t=1}^T \alpha_t \langle \hat{\param}_t - \param_t, \hat G(\param_t, \x_{t+1}) - P_{\param_t}\hat G(\param_{t}, \x_{t+1}) \rangle.
		\end{equation*}
		Separating the inner product on the right-hand side to two parts and shifting indices give us
		\begin{multline*}
			- \sum_{t=1}^T \alpha_t \langle \hat{\param}_t - \param_t, \hat G(\param_t, \x_{t+1}) - P_{\param_t}\hat G(\param_{t}, \x_{t+1}) \rangle = -\alpha_1 \langle \hat{\param}_1 - \param_1, \hat G(\param_1, \x_{2}) \rangle - \sum_{t=2}^T \alpha_t \langle \hat{\param}_t - \param_t, \hat G(\param_t, \x_{t+1}) \rangle \\
			-\alpha_T \langle\hat{\param}_T - \param_T, -P_{\param_T}\hat G(\param_T, \x_{T+1}) \rangle -\sum_{t=2}^{T} \alpha_{t-1} \langle \hat{\param}_{t-1} - \param_{t-1}, - P_{\param_{t-1}}\hat G(\param_{t-1}, \x_{t}) \rangle.
		\end{multline*}
		To bound the two sums in the right-hand side, we add and subtract $\sum_{t=2}^T \alpha_t\langle \hat{\param}_t - \param_t, P_{\param_t}\hat G(\param_t, \x_t) \rangle$ to get
		\begin{multline*}
			- \sum_{t=2}^T \alpha_t \langle \hat{\param}_t - \param_t, \hat G(\param_t, \x_{t+1}) \rangle -\sum_{t=2}^{T} \alpha_{t-1} \langle \hat{\param}_{t-1} - \param_{t-1}, - P_{\param_{t-1}}\hat G(\param_{t-1}, \x_{t}) \rangle \\
			= - \sum_{t=2}^T \alpha_t \langle \hat{\param}_t - \param_t, \hat G(\param_t, \x_{t+1}) - P_{\param_t}\hat G(\param_t, \x_t) \rangle - \sum_{t=2}^T \alpha_t \langle \hat{\param}_t-\param_t, P_{\param_t}\hat G(\param_t, \x_t) \rangle \\- \sum_{t=2}^T \alpha_{t-1} \langle \hat{\param}_{t-1} - \param_{t-1}, -P_{\param_{t-1}}\hat G(\param_{t-1}, \x_t) \rangle \\
			= - \sum_{t=2}^T \alpha_t \langle \hat{\param}_t - \param_t, \hat G(\param_t, \x_{t+1}) - P_{\param_t}\hat G(\param_t, \x_t) \rangle - \sum_{t=2}^T \alpha_t \langle \hat{\param}_t-\param_t, P_{\param_t}\hat G(\param_t, \x_t) - P_{\param_{t-1}}\hat G(\param_{t-1}, \x_t) \rangle \\
			-\sum_{t=2}^T \alpha_t\langle (\hat{\param}_{t-1} - \param_{t-1}) - (\hat{\param}_t-\param_t), -P_{\param_{t-1}}\hat G(\param_{t-1}, \x_t) \rangle - \sum_{t=2}^T (\alpha_{t-1} - \alpha_t)\langle \hat{\param}_{t-1}-\param_{t-1}, -P_{\param_{t-1}}\hat G(\param_{t-1}, \x_t) \rangle.
		\end{multline*}
		Plugging back to~\eqref{eq: ser1}, we get
		\begin{align}
			\mathbb{E} \sum_{t=1}^T \alpha_t \langle \hat{\param}_t - \param_t, \nabla f(\param_t) - G(\param_t, \x_{t+1}) \rangle &\leq \mathbb{E}\left[-\alpha_1 \langle \hat{\param}_1 - \param_1, \hat G(\param_1, \x_{2}) \rangle -\alpha_T \langle\hat{\param}_T - \param_T, -P_{\param_T}\hat G(\param_T, \x_{T+1}) \rangle\right] \notag \\
			&\quad - \mathbb{E}\sum_{t=2}^T \alpha_t \langle \hat{\param}_t - \param_t, \hat G(\param_t, \x_{t+1}) - P_{\param_t}\hat G(\param_t, \x_t) \rangle \notag\\
			&\quad - \mathbb{E}\sum_{t=2}^T \alpha_t \langle \hat{\param}_t-\param_t, P_{\param_t}\hat G(\param_t, \x_t) - P_{\param_{t-1}}\hat G(\param_{t-1}, \x_t) \rangle \notag\\
			&\quad  -\mathbb{E}\sum_{t=2}^T \alpha_t\langle (\hat{\param}_{t-1} - \param_{t-1}) - (\hat{\param}_t-\param_t), -P_{\param_{t-1}}\hat G(\param_{t-1}, \x_t) \rangle \notag\\
			&\quad- \mathbb{E}\sum_{t=2}^T (\alpha_{t-1} - \alpha_t)\langle \hat{\param}_{t-1}-\param_{t-1}, -P_{\param_{t-1}}\hat G(\param_{t-1}, \x_t) \rangle.\label{eq: www3}
		\end{align}
		We bound the right-hand side in order. First
		\begin{equation*}
			\mathbb{E}\left[-\alpha_1 \langle \hat{\param}_1 - \param_1, \hat G(\param_1, \x_{2}) \rangle -\alpha_T \langle\hat{\param}_T - \param_T, -P_{\param_T}\hat G(\param_T, \x_{T+1}) \rangle\right] \leq \frac{\alpha_1}{2} \left(\| \param_1 - \hat\param_1 \|^2 + C_1^2\right) +  \frac{2\alpha_TLC_2}{\hat\rho - \rho},
		\end{equation*}
		by~\cref{assumption:A3},~\cref{assumption:A2'} and Lemma~\ref{lem: bdd_xt_xthat}.
		
		Second, we use the tower rule and $\mathcal{F}_t$ measurability of $\param_t-\hat{\param}_t$ where $\mathcal{F}_{t}:=\sigma( X_{0}, \param_{0}, X_{1},\param_{1},\dots, X_{t},\param_{t})$, with~\cref{assumption:A1'} (used with $H(X_t) = G(\param_t, \x_t)$) to get
		\begin{equation*}
			\mathbb{E}\sum_{t=2}^T \alpha_t \langle \hat{\param}_t - \param_t, \hat G(\param_t, \x_{t+1}) - P_{\param_t}\hat G(\param_t, \x_t) \rangle = \mathbb{E}\sum_{t=2}^T \alpha_t \langle \hat{\param}_t - \param_t, \mathbb{E}[\hat G(\param_t, \x_{t+1})~|~\mathcal{F}_t] - P_{\param_t}\hat G(\param_t, \x_t) \rangle =0.
		\end{equation*}
		Third,
		\begin{align*}
			- \mathbb{E}\sum_{t=2}^T \alpha_t \langle \hat{\param}_t-\param_t, P_{\param_t}\hat G(\param_t, \x_t) - P_{\param_{t-1}}\hat G(\param_{t-1}, \x_t) \rangle &\leq \mathbb{E}\sum_{t=2}^T C_3\alpha_t \| \hat{\param}_t-\param_t\|\| \param_t - \param_{t-1}\|\\
			&\leq \mathbb{E}\sum_{t=2}^T \frac{2L^2C_3\alpha_t\alpha_{t-1}}{\hat\rho-\rho},
		\end{align*}
		where the first step used~\cref{assumption:A2'} and the last step used the definition of $\param_t$, nonexpansiveness of projection,~\cref{assumption:A3} and Lemma~\ref{lem: bdd_xt_xthat}.
		
		Fourth, we have
		\begin{align*}
			-\mathbb{E}\sum_{t=2}^T \alpha_t\langle (\hat{\param}_{t-1} - \param_{t-1}) - (\hat{\param}_t-\param_t), -P_{\param_{t-1}}\hat G(\param_{t-1}, \x_t) \rangle &\leq \mathbb{E}\sum_{t=2}^T C_2 \alpha_t  \left(\|\param_t - \param_{t-1}\| + \| \hat{\param}_t - \hat{\param}_{t-1}\| \right) \\
			&\leq \sum_{t=2}^T C_2 \alpha_t\left( \alpha_{t-1}L + \frac{\hat\rho}{\hat\rho-\rho}\alpha_{t-1} L \right).
		\end{align*}
		where the first step used~\cref{assumption:A2'}, and triangle inequality, and the last step used the definition of $\param_t$, nonexpansiveness of projection,~\cref{assumption:A3} and Lemma~\ref{lem: prox_wc_lips}.
		
		Fifth, by using Lemma~\ref{lem: bdd_xt_xthat} and~\cref{assumption:A2'}, we have
		\begin{equation*}
			- \mathbb{E}\sum_{t=2}^T (\alpha_{t-1} - \alpha_t)\langle \hat{\param}_{t-1}-\param_{t-1}, -P_{\param_{t-1}}\hat G(\param_{t-1}, \x_t) \rangle \leq \sum_{t=2}^T |\alpha_{t-1} - \alpha_t | \frac{2LC_2}{\hat\rho-\rho}.
		\end{equation*}
		Plugging these five estimations to~\eqref{eq: www3} bounds the error term in~\eqref{eq: ser1}. Then plugging this to~\eqref{eq: tfr3}, we finish the proof after following the same steps as Theorem~\ref{th: th_proj_sgd}.
	\end{proof}

	\section{Convergence of Online Dictionary Learning with first-order methods}
	\label{sec:ODL_app}

	\begin{assumption} \label{assumption:A6}
		For each $\X$ and $\param$, the function $\param\mapsto \ell(\X,\param)=\inf_{H\in \Param'} \left( d(\X,\param H) + R(H)\right)$ is $L$-smooth for some $L>0$. 
	\end{assumption}
	\noindent In \cite{mairal2010online}, it was shown that both \cref{assumption:A5} and \cref{assumption:A6} are verified when  $d$ satisfies 
	\begin{align}\label{eq:d_ODL_L2}
		d(\X,\param H) = \lVert \X - \param H \rVert_{F}^{2} + \kappa_{2} \lVert H \rVert^{2}_F + \lambda \lVert H \rVert_{1},
	\end{align}
	where $\kappa_{2}>0$ and $\lambda \ge 0$. Then the following result is a direct consequence of our main results, Theorems \ref{th: th_proj_sgd}, \ref{th: th_proj_sgd_smooth}, \ref{th: th_adagrad}, and \ref{th: th_proj_sgd_state_dep}.
	
	\begin{corollary}\label{cor: odl}
		Consider \eqref{eq:ODL_f} and assume \cref{assumption:A5}. 
		Suppose we have a sequence of data matrices $(\X_{t})_{t\ge 0}$ and let $(\param_{t})_{t\ge 1}$ be the sequence of dictionary matrices in $\Param\subseteq \R^{p\times r}$ obtained by either of the three algorithms: Projected SGD (Algorithm \ref{algorithm:prox_grad}), AdaGrad (Algorithm \ref{algorithm:adagrad}), and stochastic heavy ball (Algorithm \ref{algorithm:shb}). Suppose 
		\begin{description}
			\item[(a1)] $\Param$ is compact and the sequence of data matrices $(\X_{t})_{t\ge 0}$ satisfy the assumption \cref{assumption:A1} and has a compact support;
			\item[(a2)] For each $\X$, the function $\param\mapsto \ell(\X, \param)$ is $\rho$-smooth for some $\rho>0$ over $\Param$.
		\end{description}
		Then in all cases, 
		we sample $\hat{t}\in \{1,\dots,T\}$ and compute $\breve{\param}_{\hat t+1}$ as in Theorem \ref{th: th_proj_sgd_smooth} and have the complexity $\mathbb{E}\left[\mathrm{dist}\left(\mathbf{0}, \partial(f+\iota_{\Param})(\breve{\param}_{t+1})\right)\right] \leq \varepsilon$ with $T=\tilde{O}(\varepsilon^{-4})$ samples.     Furthermore, Projected SGD and SHB
		converges almost surely to the set of stationary point of the objective function for \eqref{eq:ODL_f}. In particular, the above results hold under \cref{assumption:A1} and when $d$ is as in \eqref{eq:d_ODL_L2}. 
	\end{corollary}

	
		
		
	
\end{document}